\documentclass[english, 10pt,reqno]{amsart}
\usepackage{geometry}  
\usepackage{amssymb,amsmath,amsthm,amsfonts,color}
\usepackage{mathrsfs,dsfont, comment,mathscinet}
\usepackage{graphicx}
\usepackage{epstopdf}
\usepackage{mathtools}
\usepackage{babel}
\usepackage{enumerate,esint}

\usepackage{indentfirst}
\usepackage{bm}
\usepackage{picinpar}
\usepackage{caption}
\usepackage{subcaption}
\usepackage{soul}

\usepackage[colorlinks=true, pdfstartview=FitV, linkcolor=blue, citecolor=blue, urlcolor=blue]{hyperref}

\usepackage[toc,page]{appendix}

\usepackage{lineno}
\newcommand*\patchAmsMathEnvironmentForLineno[1]{%
  \expandafter\let\csname old#1\expandafter\endcsname\csname #1\endcsname
  \expandafter\let\csname oldend#1\expandafter\endcsname\csname end#1\endcsname
  \renewenvironment{#1}%
     {\linenomath\csname old#1\endcsname}%
     {\csname oldend#1\endcsname\endlinenomath}}%
\newcommand*\patchBothAmsMathEnvironmentsForLineno[1]{%
  \patchAmsMathEnvironmentForLineno{#1}%
  \patchAmsMathEnvironmentForLineno{#1*}}%
\AtBeginDocument{%
\patchBothAmsMathEnvironmentsForLineno{equation}%
\patchBothAmsMathEnvironmentsForLineno{align}%
\patchBothAmsMathEnvironmentsForLineno{flalign}%
\patchBothAmsMathEnvironmentsForLineno{alignat}%
\patchBothAmsMathEnvironmentsForLineno{gather}%
\patchBothAmsMathEnvironmentsForLineno{multline}%
}

\usepackage{algorithm2e}
\usepackage{etoolbox}
\usepackage{authblk}
\usepackage{amsaddr}

\usepackage{caption}
\usepackage{subcaption}

\allowdisplaybreaks 

\mathtoolsset{showonlyrefs} 

\makeatletter

\usepackage{fancyhdr}
 
\makeatletter
\patchcmd{\@maketitle}
  {\ifx\@empty\@dedicatory}
  {\ifx\@empty\@date \else {\vskip3ex \centering\footnotesize\@date\par\vskip1ex}\fi
   \ifx\@empty\@dedicatory}
  {}{}
\patchcmd{\@maketitle}
  {\ifx\@empty\@date\else \@footnotetext{\@setdate}\fi}
  {}{}{}
\makeatother

\usepackage{scalerel}
\usepackage{stackengine,wasysym}

\begingroup
\newtheorem{theorem}{Theorem}[section]

\newtheorem{define}[theorem]{Definition}

\endgroup

\theoremstyle{definition}
\begingroup

\newtheorem{corollary}[theorem]{Corollary}
\newtheorem{lemma}[theorem]{Lemma}
\newtheorem{remark}[theorem]{Remark}

\newtheorem{proposition}[theorem]{Proposition}

\endgroup

\newcommand\ba[1]{\begin{align}\label{#1}}
\newcommand\ea{\end{align}}
\newcommand\bas{\begin{align*}}
\newcommand\eas{\end{align*}}

\newcommand\ee{\end{equation}}
\newcommand\be{\begin{equation}}
\newcommand\ees{\end{equation*}}
\newcommand\bes{\begin{equation*}}

\definecolor{Pblue}{RGB}{87,158,208}

\newcommand{\e}{\varepsilon}

\newcommand{\om}{\omega}
 
\newcommand{\R}{{\mathbb R}}

\newcommand{\rn}{{{\R}^N}}

\newcommand{\wtos}{\mathrel{\mathop{\rightharpoonup}\limits^*}}
\newcommand{\wtof}{\mathrel{\mathop{\rightharpoonup}\limits^{\mathbb F}}}

\newcommand{\CC}{{\mathcal A}}
\newcommand{\N}{{\mathbb{N}}}

\newcommand{\T}{{\mathbb T}}

\newcommand\norm[1]{\left\|#1\right\|}

\newcommand{\abs}[1]{\left\lvert#1\right\rvert} 
\newcommand{\fsp}[1]{{\left(#1\right)}}
\newcommand{\fmp}[1]{\left[#1\right]}
\newcommand{\flp}[1]{\left\{#1\right\}}  
\newcommand{\fjp}[1]{\left<#1\right>}

\newcommand{\vp}{\varphi}

\newcommand{\limn}{\lim_{n\rightarrow\infty}}

\newcommand{\divg}{{\operatorname{div\,}}}

\newcommand{\seqn}[1]{\left\{#1\right\}_{n=1}^\infty}  
\newcommand{\seqk}[1]{\left\{#1\right\}_{k=1}^\infty}  
\newcommand{\seql}[1]{\left\{#1\right\}_{l=1}^\infty}

\definecolor{CMUred}{RGB}{153,0,0}
\definecolor{CMUgreen}{RGB}{0,135,81}
\definecolor{CMUblue}{RGB}{0,51,127}
\definecolor{Pblue}{RGB}{87,158,208}

\newcommand{\argmin}{{\operatorname{arg\,min}}}

\newcommand{\dist}{{\operatorname{dist}}}

\newcommand{\ta}{{\tilde\alpha}}

\newcommand{\Tl}{{\mathbb T_l}}

\def\argmin{\mathop{\rm arg\, min}}

\DeclareRobustCommand{\om}{\omega}

\numberwithin{equation}{section}

\newcommand{\normmm}[1]{{\left\vert\kern-0.25ex\left\vert\kern-0.25ex\left\vert #1 
    \right\vert\kern-0.25ex\right\vert\kern-0.25ex\right\vert}}
    
\addto\captionsenglish{}
    
\title{Learning optimal orders of the underlying Euclidean norm in total variation image denoising}
\date{March 28, 2019}
\author[Pan Liu]{Pan Liu}
 \address[Pan Liu]{Department of Applied Mathematics and Theoretical Physics,\\ University of Cambridge, Wilberforce Road, Cambridge CB3 0WA, UK}
 \email[P. Liu] {panliu.0923@maths.cam.ac.uk}

\author{Carola-Bibiane Sch\"onlieb}
 \address[Carola Schoenlieb]{Department of Applied Mathematics and Theoretical Physics,\\ University of Cambridge, Wilberforce Road, Cambridge CB3 0WA, UK}
 \email[Carola Schoenlieb] {cbs31@cam.ac.uk}

\begin{document}



\maketitle
\begin{abstract}
A novel class of semi-norms, generalising the notion of the isotropic total variation $TV_{2}$ and the an-isotropic total variation $TV_{1}$ is introduced. A supervised learning method via bilevel optimisation is proposed for the computation of optimal parameters for this class of regularizers. Existence of solutions to the bilevel optimisation approach is proven. Moreover, a finite-dimensional approximation scheme for the bilevel optimisation approach is introduced that can numerically compute a global optimizer to any given accuracy. 
\end{abstract}
\tableofcontents
\thispagestyle{empty}

\setcounter{section}{0}
\setcounter{subsection}{0}
\setcounter{equation}{0}
\setcounter{theorem}{0}

\section{Introduction}\label{sec_introduction}
Total variation denoising is given by the minimization problem
\be\label{intro_denoisy2}
u_\alpha:=\argmin\flp{\norm{u-u_\eta}_{L^2(Q)}^2+\alpha TV(u):\,\, u\in BV(Q)},
\ee
where $u_\eta\in L^2(Q)$ denotes a given noisy image on the domain $Q:=(0,1)^N$, $N\in\N$, $\alpha\in\R^+$ denotes the \emph{regularization parameter}, and $TV(u)$ is the (isotropic) total variation defined by
\be\label{intro_tv_normal}
TV(u):=\sup\flp{\int_Q u\,\divg\vp\, dx:\,\,\vp\in C_c^\infty(Q;\,\rn),\,\,\abs{\vp}\leq 1}
\ee
which is $<+\infty$ for $u\in BV(Q)$, the space of functions of bounded variation. Here $\abs{\cdot}$ refers to the $\ell^2$-Euclidean norm, that is, for $x=(x_1,x_2)\in\R^2$ we have $\abs{x}=(x_1^2+x_2^2)^{1/2}$. The $TV$ denoising model \eqref{intro_denoisy2} is also called \emph{ROF} model, named after the pioneering paper \cite{rudin1992nonlinear} of \emph{Rudin}, \emph{Osher}, and \emph{Fatemi}. The $TV$ denoising model is known for its ability to both denoise an image and at the same time preserve discontinuities. Due to this edge-preserving property the \emph{ROF} model has established itself in the image processing literature.\\

Next to these desirable denoising properties the \emph{ROF} model, however, also comes with disadvantages. One of those is the tendency of the ROF model to generate unnecessary edges. These turn originally smoothly changing image intensities into piecewise constant intensity areas which create blocky-like artefacts also known as \emph{stair-casing}. Another disadvantage of the \emph{ROF} model is that it leads to a contrast loss near edges that mainly depends on their curvature. \\

\textbf{$\mathbf{TV_p}$, $\mathbf{1\leq p\leq\infty}$.} In this paper we consider a generalised notion of the total variation in which we replace the underlying $\ell^2$-Euclidean norm by the $\ell^p$-Euclidean norm for $1\leq p<\infty$. We therefore write $\abs{\cdot}_p$ for $\abs{\cdot}$ and $TV_p(\cdot)$ for $TV(\cdot)$. In \cite{MR2727338}, for instance, another variant of the total variation has been proposed, by switching the underlying Euclidean norm from $\ell^2$ to $\ell^1$ (\cite{MR1666943}), i.e.,
\be
TV_1(u):=\sup\flp{\int_Q u\,\divg\vp\, dx:\,\,\vp\in C_c^\infty(Q;\,\rn),\,\,\abs{\vp}_{1}^{\ast}\leq 1}<+\infty,
\ee
where $\abs{\cdot}_1$ denotes the Euclidean $1$-norm in the sense that for $x=(x_1,x_2)\in\R^2$, $\abs{x}_1=\abs{x_1}+\abs{x_2}$, and $\abs{\cdot}_1^\ast$ denotes the dual norm associated with $\abs{\cdot}_1$. 
Successful applications of $TV_1$, also called the anisotropic total variation, can be found in \cite{MR2727338, moll2005anisotropic, doi:10.1137/16M1103610, MR2496060, strong1996exact}. In particular, in \cite{MR2727338}, it has been observed that $TV_1$ has the ability to suppress the stair-casing effect which is a typical artifact induced by isotropic total variation. Total variations defined by Euclidean $\ell^p$ - norms, for $p\notin\flp{1,2}$, have rarely been analyzed, and hence their performance is largely unknown.\\
\textbf{Bilevel optimization.} The quality of a reconstructed image $u_\alpha$ obtained from \eqref{intro_denoisy2} highly depends on the choice of the regularization parameter $\alpha$. If $\alpha$ is too large then $TV(u)$ is penalized too much and the image is over-smoothed, resulting in a loss of information in the reconstructed image. On the other hand, if $\alpha$ is too small then the reconstructed image is under-regularized and noise is left in the reconstruction. Classical approaches to choose an appropriate regularisation parameter are Morozov's discrepancy principle \cite{Morozov:1966aa}, generalised cross-validation \cite{Golub:1979aa} or L-curve \cite{Hansen:1992aa} just to name a few \cite{Engl:2000aa}. A recent approach to determine the optimal $\alpha$ is bilevel optimization (see e.g.,\cite{MR1984880,tappen2007learning,domke2012generic,de2013image,kunisch2013bilevel,chen2013revisiting,chen2014insights}). Here, an optimality criterion for the denoised image is used, given in terms of a loss function for the minimiser $u_\alpha$, and an $\alpha$ found which minimises this loss. In most cases the loss function is supervised, that given a training set of noisy images $(u^i_\eta)_{i=1}^I$ and corresponding noise-free images $(u^i_c)_{i=1}^I$, bilevel optimization for the regularization parameter $\alpha$ in \eqref{intro_denoisy2} reads 
\begin{flalign}
\text{Level 1. }&\,\,\,\,\,\,\,\,\,\,\,\,\,\,\,\,\,\,\,\,\,\,\tilde \alpha\in\argmin\flp{\sum_{i=1}^I\norm{u^i_\alpha-u^i_c}_{L^2(Q)}^2:\,\, \alpha\in\R^+}, \textrm{ s.t. }\tag{{$\mathcal B$-L1}}\label{scheme_B1_BV}&\\
\text{Level 2. }&\,\,\,\,\,\,\,\,\,\,\,\,\,\,\,\,\,\,\,\,\,\,u^i_\alpha:=\argmin\flp{\norm{u-u^i_\eta}_{L^2(Q)}^2+\alpha {TV_2(u)}:\,\, u\in BV(Q)}, \quad i=1,\ldots, I.\tag{{$\mathcal B$-L2}}\label{scheme_B2_BV}&
\end{flalign}
The Level 1 problem in \eqref{scheme_B1_BV} looks for an $\alpha$ that minimizes the average $L^2$-distance between minimizers $u^i_\alpha$ of the Level 2 problem \eqref{scheme_B2_BV} and clean images $u^i_c$. It has been proven in \cite{reyes2015structure} that \eqref{scheme_B1_BV} admits at least one positive solution $\tilde\alpha\in\R^+$ provided that $TV_2(u^i_\eta)>TV_2(u^i_c)$, which is a reasonable assumption for image denoising. \\

For simplicity, in what follows we will omit the index $i$ from the training set and perform our analysis for a single pair $(u_\eta,u_c)$ of noisy and noise-free image, respectively. Everything we will discuss, however, will still hold for the case of multiple images in the training set.\\ 
%

\textbf{Bilevel optimisation for $\mathbf{TV_p}$.} For the purpose of studying $TV_p$ for $p\in [1,+\infty]$ we extend the bilevel training scheme $\mathcal B$ to scheme $\mathcal T$ as
\begin{flalign}
\text{Level 1. }&\,\,\,\,\,\,\,\,\,\,\,\,\,\,\,\,\,\,\,\,\,\,
(\alpha_\T,p_\T)\in\mathbb A[\T]:=\argmin\flp{\norm{u_{\alpha,p}-u_c}_{L^2(Q)}^2:\,\, (\alpha,p)\in\T},\tag{{$\mathcal T$-L1}}\label{scheme_B1_BV_p}&\\
\text{Level 2. }&\,\,\,\,\,\,\,\,\,\,\,\,\,\,\,\,\,\,\,\,\,\,
u_{\alpha,p}:=\argmin\flp{\norm{u-u_\eta}_{L^2(Q)}^2+\alpha {TV_p(u)}:\,\, u\in BV(Q)}\tag{{$\mathcal T$-L2}}\label{scheme_B2_BV_p},&
\end{flalign}
where, and in what follows, we call the set
\be\label{intro_training_ground}
\mathbb T:=\R^+\times[1,+\infty]
\ee
the \emph{training ground}, in which we optimize parameters $\alpha$ and $p$, and the set $\mathbb A[\T]$ the optimal set associated with $\T$, which contains the optimal parameters produced by the scheme $\mathcal T$. We point out that the new training scheme $\mathcal T$ simultaneously optimizes both the parameter $\alpha$ and the order $p$. \\\\
\textbf{Our contribution.} This paper contains two main results. The first result is contained in Theorem \ref{main_scheme_K_result} and proves that the scheme $\mathcal T$ \eqref{scheme_B1_BV_p}-\eqref{scheme_B2_BV_p} admits at least one optimal solution $(\alpha_\T,p_\T)\in\T$. This existence result is based on Theorem \ref{thm:new-Gamma} where we show that the functionals 
\be
 \mathcal I_{\alpha,p}(u):=\norm{u-u_\eta}_{L^2(Q)}^2+\alpha TV_{p}(u)\text{ for every }u\in BV(Q)
 \ee
are continuous, in the sense of $\Gamma$-convergence in the weak* topology of $BV(Q)$ (see \cite{braides2002gamma, MR1201152}), with respect to the parameters $(\alpha,p)$. We prove this by showing that the collection of new spaces, induced from $TV_p$ semi-norms, itself exhibits certain compactness and lower semicontinuity properties. \\\\
Our second contribution is a proposal for how to numerically determine the optimal solution of scheme $\mathcal T$, or equivalently compute global minimizers of the \emph{assessment function} $\mathcal A(\alpha,p)$: $\T\to\R^+$ defined as 
\be\label{cost_map_intro}
\mathcal A(\alpha,p):=\norm{u_{\alpha,p}-u_c}_{L^2(Q)}^2,
\ee
where $u_{\alpha,p}$ is obtained from \eqref{scheme_B2_BV_p}. We note that computing such global minimizers would be straightforward if $\mathcal A(\alpha,p)$ is quasi-convex in the sense of \cite{MR1819784}, or simply convex. In this case Newton's descent method or line search could be applied to compute a the global minimizer. However, as we shall later show in Figure \ref{fig:2nd_counter_exam} even for a fixed $p_0\in[1,+\infty]$ the assessment function $\mathcal A(\alpha,p_0)$ is not quasi-convex, and hence those methods mentioned above might get trapped in a local minimum. To overcome this difficulty, we introduce the concept of an \emph{acceptable optimal solution} that is a solution to $\mathcal T$ with a prescribed error. To be precise, we say the solution $(\bar \alpha,\bar p)$ is acceptable with error $\e>0$ if 
\be\label{accptable_train_result_intro}
\abs{\CC(\bar \alpha,\bar p)-\CC(\alpha_\T, p_\T)}<\e,
\ee
where $(\alpha_\T, p_\T)\in\mathbb A[\T]$ is a 
globally optimal solution obtained from the scheme $\mathcal T$.\\\\7
For computing such an acceptable optimal solution, we propose in Section \ref{thaos_sec} a finite  approximation method. We construct a sequence of finite sets $\T_l$ indexed by $l\in\N$, such that $\#\flp{\T_l}<+\infty$. For the precise definition of $\T_l$ we refer to Definition \ref{finite_playground_l}. We point out here that, since $\#\flp{\T_l}<+\infty$, the optimal solution(s)
\be
(\alpha_\Tl, p_\Tl)\in\mathbb A[\Tl]:=\argmin\flp{\CC(\alpha,p):\,\, (\alpha,p)\in\T_l}
\ee
can be determined precisely by evaluating $\CC(\alpha,p)$ at each $(\alpha,p)\in  \T_l$. From there, it is not hard to prove that $(\alpha_\Tl, p_\Tl)\to( \alpha_\T, p_\T)$ and $\mathcal A(\alpha_\Tl, p_\Tl)\to \mathcal A(\alpha_\T, p_\T)$, as $l\to\infty$, by using standard $\Gamma$-convergence techniques. This is, however, still not enough to allow the computation of an acceptable solution as in \eqref{accptable_train_result_intro}. To achieve such a result, we prove in Theorem \ref{article_result} an estimate for a fixed index $l\in\N$, which gives an estimate of the form
\be\label{error_est_intro}
\abs{\CC(\alpha_\Tl, p_\Tl)-\CC( \alpha_\T, p_\T)}\leq  \sqrt{\alpha_U}\fmp{\fsp{1/l}^{1/2}+2{\fsp{1-N^{-1/\sqrt l}}}^{1/2}}(TV_1(u_\eta))^{1/2},
\ee
in which $\alpha_U$ can be determined numerically (see Proposition \ref{prop:last_word_go_upper}). 
Therefore, by using estimate \eqref{error_est_intro}, we can acquire the desired index $l\in \N$ so that the associated optimal solution $(\alpha_\Tl, p_\Tl)\in\T_l$ is an acceptable optimal solution for the error $\e>0$.\\\\
We note that the estimate \eqref{error_est_intro} requires that $u_\eta\in BV(Q)$, which usually does not hold for a noisy image $u_\eta$. To overcome this, in Section \ref{relaxation_corrupted_sec} we show that, for any given $\e>0$, even if $u_\eta\in L^2(Q)\setminus BV(Q)$, we are still able to find $l\in \N$ such that 
\be
\abs{\CC(\alpha_\Tl, p_\Tl)-\CC( \alpha_\T, p_\T)}\leq \e,
\ee
i.e., the associated optimal solution $(\alpha_\Tl, p_\Tl)\in\T_l$ is an acceptable optimal solution for the error $\e>0$. We do so by introducing a piece-wise constant approximation of the corrupted image $u_\eta\in L^2(Q)\setminus BV(Q)$, and we refer readers to Theorem \ref{article_result_finite_resolution} and Corollary \ref{main_combine_coro} for details.\\\\
\textbf{Organisation of the paper.} The paper is organized as follows. In Section \ref{sec_speed_estimation_pre} we collect some notations and preliminary results. The $\Gamma$-convergence and the bilevel training scheme are the subjects of Sections \ref{GOTLITF} and \ref{sec_existence_training}, respectively. Section \ref{thaos_sec} is devoted to the analysis of the finite approximation training scheme and the proof of Theorem \ref{article_result}. Finally, in Section \ref{sec_NSC} some numerical simulations and insights. 

\section{The $\ell^p$-anisotropic total variation, $\Gamma$-convergence, and an optimal training scheme}
We recall that, throughout this article, $u_\eta\in L^2(Q)$ denotes a given datum representing a noisy image, $u_c\in BV(Q)$ represents the corresponding noise-free image, and $u_{\alpha,p}\in BV(Q)$ is the reconstructed image obtained from \eqref{scheme_B2_BV_p} for a given set $(\alpha,p)\in\T$.
\subsection{The $\ell^p$-(an)-isotropic total variation}\label{sec_speed_estimation_pre}
We recall from \cite{evans2015measure} that a function $u\in L^1(Q)$ has \emph{bounded variation} in $Q$ if 
\be
TV_2(u):=\sup\flp{\int_Q u\,\divg\vp\, dx:\,\,\vp\in C_c^\infty(Q;\,\rn),\,\,\abs{\vp}_2\leq 1}<+\infty,
\ee
and write $BV(Q)$ to denote the \emph{space of functions of bounded variation}. We also define the norm
\be
\norm{u}_{BV(Q)}:=\norm{u}_{L^1(Q)}+TV_2(u).
\ee
We next define the Euclidean $\ell^p$-norm for $p\in[1,+\infty]$ and for $x=(x_1,x_2,\ldots,x_N)\in \rn $ as  
\be\label{p_euclidean_def}
\abs{x}_p:=(\abs{x_1}^p+\abs{x_2}^p+\cdots+\abs{x_N}^p)^{1/p}.
\ee
We recall that $\abs{\cdot}_p$ for $p\in[1,+\infty]$ are equivalent norms on $\rn$. To be precise, for any $1\leq p_1<p_2\leq \infty$ and $x\in\rn$, we have that
\be\label{equivalent_p_norm}
\abs{x}_{p_2}\leq\abs{x}_{p_1}\leq N^{1/p_1-1/p_2}\abs{x}_{p_2}.
\ee
\begin{define}[The $\ell^p$-an-isotropic total variation]
Let $u\in L^1(Q)$ be given, we define, for $1\leq p\leq +\infty$, the $\ell^p$ an-isotropic total variation $TV_p$ by
\be\label{est_vpiso_de_p}
TV_p(u):=\sup\flp{\int_Q u\,\divg\vp\, dx:\,\,\vp\in C_c^\infty(Q;\,\rn),\,\,\abs{\vp}_{p}^{\ast}\leq 1},
\ee
where $\abs{\cdot}_p^\ast$ denotes the dual norm associated with $\abs{\cdot}_p$. 
\end{define}
\begin{remark}\label{equivalent_semi}
In view of \eqref{equivalent_p_norm}, we have that the $TV_p$ semi-norms, for $1\leq p\leq\infty$, are equivalent. That is, for $1\leq p_1<p_2\leq +\infty$, we have that 
\be\label{equivalent_semi_TV_p}
TV_{p_2}(u)\leq TV_{p_1}(u)\leq  N^{1/p_2^\ast-1/p_1^\ast}TV_{p_2}(u),
\ee
for all $u\in BV(Q)$. In particular, we have 
\be\label{equivalent_semi_TA_all}
TV(u)= TV_2(u)\leq N\cdot TV_p(u)
\ee
for any $p\in[1,+\infty]$ and $u\in BV(Q)$.
\end{remark}
\subsection{$\Gamma$-convergence of functionals defined by $TV_p$ seminorms}\label{GOTLITF}
Let $p\in[1,+\infty]$ and $\alpha\in\R^+$. We define the functional $\mathcal I_{\alpha,p}$: $L^2(Q)\to [0,+\infty]$ as
\be
\mathcal I_{\alpha,p}(u):=
\begin{cases}
\norm{u-u_\eta}_{L^2(Q)}^2+\alpha TV_p(u),&\text{ if }u\in BV(Q)\\
+\infty, &\text{ otherwise }.
\end{cases}
\ee
The following theorem is the main result of Section \ref{GOTLITF}.
\begin{theorem}\label{thm:new-Gamma}
Let $\seqn{p_n}\subset [1,+\infty]$ and $\seqn{\alpha_n}\subset \R^+$ be given such that $p_n\to p_0$ and $\alpha_n\to \alpha_0\in\R^+$. Then the functional $\mathcal I_{\alpha_n,p_n}$ $\Gamma$-converges to $\mathcal I_{\alpha_0,p_0}$ in the weak* topology of $BV(Q)$. Namely, for every $u\in BV(Q)$ the following two assertions hold:
\begin{enumerate}
\item[{\rm (LI)}] If 
\be
 u_n\wtos u_0\text{ weakly* in }BV(Q),\ee
then 
\be
\liminf_{n\to +\infty}\mathcal I_{\alpha_n,p_n}(u_n)\geq \mathcal I_{\alpha_0,p_0}(u_0).
\ee
\item[{\rm (RS)}]
For each $u_0\in BV(Q)$, there exists $\seqn{u_n}\subset BV(Q)$ such that 
\be u_n\wtos u_0\text{ weakly}^\ast\text{ in }BV(Q),\ee
and 
\be
\limsup_{n\to +\infty}\,\mathcal I_{\alpha_n,p_n}(u_n)\leq \mathcal I_{\alpha_0,p_0}(u_0).
\ee
\end{enumerate}
\end{theorem}
We subdivide the proof of Theorem \ref{thm:new-Gamma} into two propositions.
\begin{proposition}[$\Gamma$-$\liminf$ inequality]\label{compact_semi_para}
Let $\seqn{p_n}\subset [1,+\infty]$ and $\seqn{\alpha_n}\subset \R^+$ be such that $p_n\to p$ and $\alpha_n\to \alpha\in\R^+$.
Let $\seqn{u_n}\subset BV(Q)$ be such that
\be\label{BGV_up_bdd}
\sup\flp{ \mathcal I_{\alpha_n,p_n}(u_n):\,\,n\in\mathbb N}<+\infty.
\ee
Then, there exists $u\in BV(Q)$ such that, up to the extraction of a (non-relabeled) subsequence, there holds
\be
u_n\wtos u \text{ in }BV(Q),
\ee
with
\be
\label{star}
\liminf_{n\to +\infty} TV_{p_n}(u_n)\geq TV_p(u).
\ee
\end{proposition}
\begin{proof}
We prove the statement for $\alpha_n\equiv1$ only, as the general case for $\alpha\in\R^+$ can be argued with straightforward adaptations. \\\\
By \eqref{equivalent_semi_TA_all} we always have
\be
TV(u_n)\leq N\cdot TV_{p_n}(u_n) \leq N\cdot \mathcal I_{1,p_n}(u_n).
\ee
Thus, by \eqref{BGV_up_bdd} we have
\be
\sup\flp{\norm{u_n}_{BV(Q)}:\,\,n\in\mathbb N}<+\infty,
\ee
which implies that there exists $u\in BV(Q)$ such that, up to extract a subsequence (not relabeled),
\be\label{target_conv_gamma}
u_n\wtos u\text{ in }BV(Q)\text{ and }u_n\to u \text{ in }L^1 \text{ and }a.e..
\ee
Therefore, we conclude that 
\be
\liminf_{n\to\infty}TV_{p_n}(u_n)\geq \liminf_{n\to\infty} N^{-\abs{1/p_n-1/p}} TV_p(u_n)\geq TV_p(u),
\ee
where in the first inequality we used \eqref{equivalent_semi_TV_p}. This concludes the proof of \eqref{star} and hence the proposition.
\end{proof}
\begin{lemma}\label{b_differential_aa}
Let $p\in[1,+\infty]$ be fixed. Then the function
\be
f(\alpha):=TV_p(u_{\alpha,p}),\,\,\alpha\in\R^+
\ee
is continuous and monotonically decreasing and the function
\be
g(\alpha):=\norm{u_{\alpha,p}-{u_\eta}}_{L^2(Q)},\,\,\alpha\in\R^+
\ee
is continuous and non-decreasing.
\end{lemma}
\begin{proof}
We notice that by Proposition \ref{compact_semi_para}, we have $f(\alpha)$ and $g(\alpha)$ are continuous. Next, for any $0\leq\alpha_1<\alpha_2<+\infty$, by minimality there holds
\be\label{CTD_l1}
\norm{u_{\alpha_1,p}-u_\eta}_{L^2(Q)}^2+\alpha_1 TV_p\fsp{u_{\alpha_1,p}}\leq \norm{u_{\alpha_2,p}-u_\eta}_{L^2(Q)}^2+\alpha_1 TV_p\fsp{u_{\alpha_2,p}}
\ee
and 
\be\label{CTD_l2}
\norm{u_{\alpha_2,p}-u_\eta}_{L^2(Q)}^2+\alpha_2 TV_p\fsp{u_{\alpha_2,p}}\leq \norm{u_{\alpha_1,p}-u_\eta}_{L^2(Q)}^2+\alpha_2 TV_p\fsp{u_{\alpha_1,p}}.
\ee
Adding up the previous two inequalities yields
\be
\fsp{\alpha_2-\alpha_1} TV_p\fsp{u_{\alpha_2,p}}\leq \fsp{\alpha_2-\alpha_1} TV_p\fsp{u_{\alpha_1,p}},
\ee
which implies that 
\be\label{CTD_l3}
f(\alpha_2)= TV_p\fsp{u_{\alpha_2,p}}\leq  TV_p\fsp{u_{\alpha_1,p}}=f(\alpha_1).
\ee
Moreover, in view of \eqref{CTD_l1} and \eqref{CTD_l3}, we obtain that 
\begin{align}
&\norm{u_{\alpha_1,p}-u_\eta}_{L^2(Q)}^2+\alpha_1 TV_p\fsp{u_{\alpha_1,p}} \\
&\leq\norm{u_{\alpha_2,p}-u_\eta}_{L^2(Q)}^2+\alpha_1 TV_p\fsp{u_{\alpha_2,p}}\\
&\leq \norm{u_{\alpha_2,p}-u_\eta}_{L^2(Q)}^2+\alpha_1 TV_p\fsp{u_{\alpha_1,p}},
\end{align}
which, in turn, yields $\norm{u_{\alpha,p}-u_\eta}_{L^2(Q)}^2$ is non-decreasing and we are done.
\end{proof}
\begin{proposition}\label{new_equal}
Let $\seqn{p_n}\subset [1,+\infty]$ and $\seqn{\alpha_n}\subset \R^+$ be such that $p_n\to p$ and $\alpha_n\to\alpha\in\R^+$. Then for every $u\in BV(Q)$ there holds
\be
\limsup_{n\to\infty}\alpha_nTV_{p_n}(u)= \alpha TV_p(u).
\ee
\end{proposition}
\begin{proof}
For simplicity, we only analyze this proposition under assumption $\alpha_n=1$ for all $n\in\N$. All arguments also hold for a general sequence $\seqn{\alpha_n}$ since $\alpha\in\R^+$.\\\\
The liminf inequality
\be
\liminf_{n\to\infty} TV_{p_n}(u)\geq TV_p(u)
\ee
is a direct consequence of Proposition \ref{compact_semi_para} by choosing $u_n:=u$. Next, by \eqref{equivalent_semi_TV_p} we have that 
\be
TV_{p_n}(u)\leq N^{\abs{1/p_n-1/p}} TV_p(u).
\ee
and the limsup inequality 
\be
\limsup_{n\to\infty} TV_{p_n}(u)\leq TV_p(u).
\ee
is asserted by sending $p_n\to p$.
\end{proof}

\begin{proof}[Proof of Theorem \ref{thm:new-Gamma}]
Let $(\alpha_n,p_n)\to(\alpha_0,p_0)\in\R^+\times [1,+\infty]$ be given. We obtain Property {\rm (LI)} in view of Proposition \ref{compact_semi_para}. Property {\rm (RS)} follows by Proposition \ref{new_equal}, choosing $u_n=u_0$ for every $n\in \N$.
\end{proof}
%
%
%
%
%
%
%
%
\subsection{Bilevel training scheme $\mathcal T$ and existence of solutions}\label{sec_existence_training}
We recall the training ground $\T$ from \eqref{intro_training_ground} and two levels of the scheme $\mathcal T$ are
\begin{flalign}
\text{Level 1. }&\,\,\,\,\,\,\,\,\,\,\,\,\,\,\,\,\,\,\,\,\,\,( \alpha_\T, p_\T)\in\mathbb A[\T]:=\argmin\flp{\norm{u_{\alpha,p}-u_c}_{L^2(Q)}^2:\,\, (\alpha,p)\in\T},\tag{{$\mathcal T$-L1}}\label{scheme_B1_BV_p_main}&\\
\text{Level 2. }&\,\,\,\,\,\,\,\,\,\,\,\,\,\,\,\,\,\,\,\,\,\,u_{\alpha,p}:=\argmin\flp{\norm{u-u_\eta}_{L^2(Q)}^2+\alpha {TV_p(u)}:\,\, u\in BV(Q)}\tag{{$\mathcal T$-L2}}\label{scheme_B2_BV_p_main}.&
\end{flalign}
The following theorem is the main result of Section \ref{sec_existence_training}.
\begin{theorem}\label{main_scheme_K_result}
Let $u_c\in BV(Q)$ and $u_\eta\in L^2(Q)$ be given such that 
\be\label{req_main_scheme_K_result}
TV_\infty(u_\eta)> TV_1(u_c).
\ee
Then, the training scheme $\mathcal T$ \eqref{scheme_B1_BV_p_main}-\eqref{scheme_B2_BV_p_main} admits at least one solution $(\alpha_\T, p_\T)\in (0,\alpha_U]\times[1,+\infty]$, where the upper bound  $\alpha_U\in\R^+$ is determined in Proposition \ref{prop:last_word_go_upper}.
\end{theorem}
%
%
%
%
%
%
\begin{proposition}\label{prop:last_word_go} Let $u_c$, $u_\eta\in L^2(Q)$ be given such that \eqref{req_main_scheme_K_result} holds. Then, there exists an $\alpha_L\in \R^+$ such that
\be\label{alpha_l_result}
\sup\flp{\norm{u_{\alpha_L,p} - u_c}_{L^2(Q)}^2:\,\, p\in[1,+\infty]}<\norm{u_\eta - u_c}_{L^2(Q)}^2.
\ee
\end{proposition}
\begin{proof}
Fix $\alpha>0$ and let $\partial TV_p(u)$ denotes the sub-differential of $TV_p$ at $u$, we observe that 
\begin{align}\label{way_compute_drop_0}
&\norm{u_\eta - u_c}_{L^2(Q)}^2-\norm{u_{\alpha,p} - u_c}_{L^2(Q)}^2 \\
&= 2\fjp{u_\eta-u_{\alpha,p},u_\alpha-u_c}_{L^2}+\norm{u_\eta-u_{\alpha,p}}_{L^2(Q)}^2\\
&= 2\alpha\fjp{\partial TV_p(u_{\alpha,p}),u_{\alpha,p}-u_c}_{L^2}+\norm{u_\eta-u_{\alpha,p}}_{L^2(Q)}^2\\
&\geq 2\alpha\fsp{TV_p(u_{\alpha,p}) -TV_p(u_c)}+\norm{u_\eta-u_{\alpha,p}}_{L^2(Q)}^2,
\end{align}
where at the last inequality we used the property of sub-gradient operator (see \cite[Proposition 5.4, page 24]{MR1727362}). \\\\
Recall from Lemma \ref{b_differential_aa} that $TV_p(u_{\alpha,p})$ is continuously decreasing with respect to $\alpha$, and thus we can find $\alpha_p>0$, might depend on $p$, such that 
\begin{align}\label{first_choose_eq}
&TV_p(u_c)+[TV_\infty(u_\eta)-TV_1(u_c)]/4\\
&>TV_p(u_{\alpha_p,p}) >TV_p(u_c)+[TV_\infty(u_\eta)-TV_1(u_c)]/8,
\end{align}
provided that \eqref{req_main_scheme_K_result} holds. Therefore, by \eqref{way_compute_drop_0} we have
\begin{align}\label{jianyanjiasheal_dds}
&\norm{u_\eta - u_c}_{L^2(Q)}^2-\norm{u_{\alpha_p,p} - u_c}_{L^2(Q)}^2\\
&\geq 2\alpha\fsp{TV_p(u_{\alpha_p,p}) -TV_p(u_c)}+\norm{u_\eta-u_{\alpha_p,p}}_{L^2(Q)}^2\\
&\geq 2\alpha_p [TV_\infty(u_\eta)-TV_1(u_c)]/8+\norm{u_\eta-u_{\alpha_p,p}}_{L^2(Q)}^2>0.
\end{align}
We next claim that
\be\label{jianyanjiasheal_d}
\alpha_L:=\inf\flp{\alpha_{p}:\,\, p\in[1,+\infty]}>0.
\ee
Assume that not, that is, there exists sequence $p_n\to  p\in[1,+\infty]$ such that 
\be\label{jianyanjiashe}
\limn \alpha_{{p_n}}\searrow0.
\ee
We claim that $u_{\alpha_{p_n},p_n}\to u_\eta$ strongly in $L^2$. Let $\seqk{u_{\eta,k}}\subset C^\infty(\bar Q)$ be such that $u_{\eta,k}\to u_\eta$ strongly in $L^2$, and by the optimality condition of $u_{\alpha_{p_n},p_n}$, we deduce that
\begin{align}\label{goes_to_noisy}
&\norm{u_\eta-u_{\alpha_{p_n},p_n}}_{L^2(Q)}^2+\alpha_{p_n}TV_{p_n}(u_{\alpha_{p_n},p_n})\\
&\leq \norm{u_\eta-u_{\eta,k}}_{L^2(Q)}^2+\alpha_{p_n}TV_{p_n}(u_{\eta,k})\\
&\leq \norm{u_\eta-u_{\eta,k}}_{L^2(Q)}^2+\alpha_{p_n}TV_1(u_{\eta,k}).
\end{align}
Thus, by \eqref{jianyanjiashe} and letting $n\to\infty$ first and $k\to \infty$ second, we conclude that 
\begin{align}\label{goes_to_noisy2}
&\limsup_{k,n\to\infty}\norm{u_\eta-u_{\alpha_{p_n},p_n}}_{L^2(Q)}^2+\alpha_{p_n}TV_{p_n}(u_{\alpha_{p_n},p_n})\\
&\leq \limsup_{k\to\infty}\norm{u_\eta-u_{\eta,k}}_{L^2(Q)}^2=0.
\end{align}
That is, we have $u_{\alpha_{p_n},p_n}\to u_\eta$ strongly in $L^2(Q)$ and, upon extracting a further subsequence (not relabeled), there holds $p_n\to p$ and
\be
\liminf_{n\to\infty} TV_{p_n}(u_{\alpha_{p_n},p_n})\geq TV_{ p}(u_\eta)>TV_p(u_c)+[TV_\infty(u_\eta)-TV_1(u_c)]/4,
\ee
which contradicts \eqref{first_choose_eq}. This completes the proof of \eqref{jianyanjiasheal_d}.\\\\
Now we prove \eqref{alpha_l_result}. In view of \eqref{jianyanjiasheal_dds} and \eqref{jianyanjiasheal_d} we have
\begin{align}
&\norm{u_\eta-u_{\alpha_p,p}}_{L^2(Q)}^2\\
&\leq \norm{u_\eta - u_c}_{L^2(Q)}^2-2\alpha_p [TV_\infty(u_\eta)-TV_1(u_c)]/8\\
&\leq  \norm{u_\eta - u_c}_{L^2(Q)}^2-2\alpha_L [TV_\infty(u_\eta)-TV_1(u_c)]/8,
\end{align}
and thus we conclude \eqref{alpha_l_result} since the right hand side of above inequality does not depends on $p$.
\end{proof}

Next, we determine a uniform upper bound on tha optimal regularization parameter $\alpha_\T$. We start with the following lemma, where $(u_\eta)_Q$ denotes the average of $u_\eta$ over $Q$, i.e. 
\be
(u_\eta)_Q:=\fint_Q u_\eta\,dx.
\ee
\begin{lemma}\label{stopping_2d}
Let $p\in[1,+\infty]$ be fixed and $u_{alpha,p}$ the minimiser of \eqref{scheme_B2_BV_p_main}. Then there exists $\alpha_{U_p}=\alpha_{U_p}(u_\eta)<+\infty$ such that 
\be
TV_p(u_{\alpha,p})>0\text{ for all }\alpha<\alpha_{U_p} 
\ee
and
\be\label{strictly_upp_bdd_zero}
u_{\alpha,p}=(u_\eta)_Q\text{ for all }\alpha\geq \alpha_{U_p}.
\ee
\end{lemma}
\begin{proof}
Since $p\in[1,+\infty]$ is fixed, we abbreviate $TV_p$, $u_{\alpha,p}$, and $\partial TV_p$, by $TV$,  $u_\alpha$, and $\partial TV$, respectively, in this proof. We note that the null space \be\label{space_of_constant}
\mathcal N(TV) = \flp{u\in L^1(Q),\,\, TV(u)=0},
\ee
of the total variation semi-norm is the space of constant functions (see, e.g., \cite{ambrosio2000functions}), which is a linear subspace of $L^1(Q)$. Let $\mathbb P[\cdot]$ denote the projection operator onto $\mathcal N(TV)$, and thus  $\mathbb P[u_\eta]$ is a constant by \eqref{space_of_constant}. We claim that 
\be\label{nonempty_interior}
\frac1\alpha\fsp{u_\eta-\mathbb P[u_\eta]}\in \partial TV(0)
\ee
for $\alpha>0$ large enough. Indeed, since $\partial TV(0)$ has nonempty relative interior in $\mathcal N(TV)$ (see, e.g., \cite{meyer2001oscillating}), we have that \eqref{nonempty_interior} holds for $\alpha\in\R^+$ sufficiently large since $u_\eta \in L^2(Q)$ and $\mathbb P[u_\eta]$ is a constant. Let $\alpha_0>0$ be large enough such that \eqref{nonempty_interior} hold. Then we have 
\be
\frac1{\alpha_0}\fsp{u_\eta-\mathbb P[u_\eta]}\in \partial TV(0) = \partial TV(\mathbb P[u_\eta]),
\ee
where in the last inequality we used again the fact that $\mathbb P[u_\eta]$ is a constant. That is, we have 
\be
\frac1{\alpha_0}\fsp{u_\eta-\mathbb P[u_\eta]}\in \partial TV(\mathbb P[u_\eta]),
\ee
and hence $\mathbb P[u_\eta]$ satisfies optimal condition of \eqref{scheme_B2_BV_p_main} and we conclude that $\mathbb P[u_\eta]=u_{\alpha_0}$. Therefore, we have $u_{\alpha_0}$ is a constant.\\\\
We claim next that $u_{\alpha_0}=(u_\eta)_Q$. Again by optimality condition we have
\be
\norm{u_{\alpha_0}-u_\eta}_{L^2(Q)}^2+\alpha TV(u_{\alpha_0})\leq \norm{(u_\eta)_Q-u_\eta}_{L^2(Q)}^2,
\ee
that is
\be\label{jiushizheli}
\int_Q\abs{u_{\alpha_0}-u_\eta}^2dx\leq \int_Q\abs{u_\eta-\fsp{u_\eta}_Q}^2dx.
\ee
Note that for $\lambda\in \R$,
\be
\frac d{d\lambda} \int_Q \abs{\lambda-u_\eta}^2dx =2 \int_Q\fsp{\lambda-u_\eta}dx,
\ee
which implies that the left hand side of \eqref{jiushizheli} reaches the minimum value at $\lambda=(u_\eta)_Q$. Thus, we have $u_{\alpha_0}=(u_\eta)_Q$ and we deduce that $u_\alpha=(u_\eta)_Q$ for all $\alpha\geq \alpha_0$. \\\\
Define 
\be
\alpha_{U_p}:=\inf\flp{\alpha>0,\,\, u_\alpha=(u_\eta)_Q},
\ee
and let $\seqn{\alpha_n}\subset\flp{\alpha>0,\,\, u_\alpha=(u_\eta)_Q}$ be such that $\alpha_n\searrow\alpha_{U_p}$. Thus, in view of Theorem \ref{thm:new-Gamma}, we conclude that $u_{\alpha_{U_p}}=(u_\eta)_Q$, and hence the claim is true.
\end{proof}

\begin{proposition}\label{prop:last_word_go_upper}
Let $u_c$, $u_\eta\in L^2(Q)$ be given such that \eqref{req_main_scheme_K_result} hold. Then, there exists $\alpha_U\in \R^+$ such that the following assertions hold.
\begin{enumerate}[1.]
\item
For all $\alpha\geq \alpha_U/2$ and $p\in[1,+\infty]$, we have
\be\label{alpha_r_result}
TV_p(u_{\alpha,p})=0\text{ and }u_{\alpha,p}=(u_\eta)_Q.
\ee
\item
The value of $\alpha_U$ can be determined numerically.
\end{enumerate}
\end{proposition}
\begin{proof}
For each $p\in[1,+\infty]$, let $\alpha_{U_p}>0$ be obtained from Lemma \ref{stopping_2d}. We claim that 
\be
\sup\flp{\alpha_{U,p}:\,\, p\in[1,+\infty]}<+\infty.
\ee
Take two arbitrary $p_1$ and $p_2$ such that $1\leq p_1\leq p_2\leq +\infty$. For $\alpha>0$ fixed, we have by optimality condition of \eqref{scheme_B2_BV_p_main} that 
\be
\norm{u_\eta-u_{\alpha,p_1}}_{L^2(Q)}^2+\alpha TV_{p_1}(u_{\alpha,p_1})
\leq \norm{u_\eta-u_{\alpha,p_2}}_{L^2(Q)}^2+\alpha TV_{p_1}(u_{\alpha,p_2})
\ee
and 
\be
\norm{u_\eta-u_{\alpha,p_2}}_{L^2(Q)}^2+\alpha TV_{p_2}(u_{\alpha,p_2})
\leq \norm{u_\eta-u_{\alpha,p_1}}_{L^2(Q)}^2+\alpha TV_{p_2}(u_{\alpha,p_1}).
\ee
Summing the above two inequalities yields
\begin{align}
0&\leq TV_{p_1}(u_{\alpha,p_1})-TV_{p_2}(u_{\alpha,p_1})\\
&\leq TV_{p_1}(u_{\alpha,p_2})-TV_{p_2}(u_{\alpha,p_2})\\
&\leq (N^{1/p_2^\ast-1/p_1^\ast}-1)TV_{p_2}(u_{\alpha,p_2}),
\end{align}
where at the first and last inequality we used Remark \ref{equivalent_semi}. Thus, by \eqref{strictly_upp_bdd_zero} and letting $\alpha=\alpha_{U_{p_2}}$, we infer that 
\be
0\leq TV_{p_1}(u_{\alpha_{U_{p_2}},p_1})-TV_{p_2}(u_{\alpha_{U_{p_2}},p_1})\leq0,
\ee
which, in turn, yields
\be\label{equal_zero_only}
TV_{p_1}(u_{\alpha_{U_{p_2}},p_1})= TV_{p_2}(u_{\alpha_{U_{p_2}},p_1}).
\ee
By Remark \ref{equivalent_semi} again, we have \eqref{equal_zero_only} holds unless $u_{\alpha_{U_{p_2}},p_1}\in \mathcal N(TV_{p_2})$, which implies that $u_{\alpha_{U_{p_2}},p_1}$ must be a constant. Hence, by the argument used in Lemma \ref{stopping_2d} we conclude that 
\be
u_{\alpha_{U_{p_2}},p_1}=(u_\eta)_Q\text{ and }\alpha_{U_{p_2}}\geq \alpha_{U_{p_1}}.
\ee
Therefore, we have
\be
\sup\flp{\alpha_{U_p}:\,\, p\in[1,+\infty]}\leq \alpha_{U_{\infty}}<+\infty,
\ee
and we conclude Assertion 1 by letting $\alpha_U:=2 \alpha_{U_\infty}$.\\\\
We notice that, by Lemma \ref{b_differential_aa} again, the function
\be
f(\alpha)=TV_{+\infty}(u_{\alpha,+\infty})
\ee
is continuous monotone decreasing and $f(\alpha_{U,+\infty})=0$. Hence, we can apply Newton descent to compute $\alpha_{U_\infty}$ numerically, which concludes Assertion 2.
\end{proof}
We are now ready to proof Theorem \ref{main_scheme_K_result}.
\begin{proof}[Proof of Theorem \ref{main_scheme_K_result}]
Let $u_c\in BV(Q)$ and $u_\eta\in L^2(Q)$ be given such that \eqref{req_main_scheme_K_result} holds, and recall the definition of the training ground $\T$, the assessment function $\CC(\alpha,p)$, and the optimal set $\mathbb A[\T]$ from \eqref{intro_training_ground}, \eqref{cost_map_intro}, and \eqref{scheme_B1_BV_p_main}. Let 
\be\label{eq_def_inimum_value}
m_{\T}:=\inf\flp{\norm{u_{\alpha,p}-u_c}_{L^2(Q)}:\,\, (\alpha,p)\in \T}.
\ee
We claim first that $\mathbb A[\T]$ is not empty. Let $\seqn{(\alpha_n,p_n)}\subset \T$ be a minimizing sequence obtained from \eqref{scheme_B1_BV_p_main} such that 
\be
\limn \norm{u_{\alpha_n,p_n}-u_c}_{L^2(Q)} = m_\T.
\ee
Then, up to a subsequence, there exists $(\tilde\alpha,\tilde p)\in [0,+\infty]\times [1,+\infty]$ such that $(\alpha_n,p_n)\to(\ta,\tilde p)$. Suppose for a moment that $\ta\in(0,+\infty]$. Then, in view of Theorem \ref{thm:new-Gamma} and the properties of $\Gamma$-convergence, we have
\be
u_{\alpha_n,p_n}\wtos u_{\ta,\tilde p}\text{ weakly}^\ast\text{ in }BV(Q)\text{ and strongly in }L^1(Q).
\ee
Thus, we conclude that 
\be
\norm{u_{\ta,\tilde p}-u_c}_{L^2(Q)}\leq \liminf_{n\to\infty}\norm{u_{\alpha_n,p_n}-u_C}_{L^2(Q)}=m_\T,
\ee
which implies $(\tilde\alpha,\tilde p)\in \mathbb A[\T]$. \\\\
Now we claim that $\inf\alpha_n> 0$. Indeed, assume by contradiction that $\alpha_n\searrow 0$, and in this case we already showed in \eqref{goes_to_noisy2} that $u_{\alpha_n,p_n}\to u_\eta$ in $L^2$ strong. Therefore, we have that
\be
m_\T=\liminf_{n\to\infty}\norm{u_{\alpha_n,p_n}-u_c}_{L^2(Q)}^2\geq \norm{u_\eta-u_c}_{L^2(Q)}^2,
\ee
which contradicts Proposition \ref{prop:last_word_go}. Thus, we conclude that $\inf\alpha_n>0$, which implies $\ta>0$.\\\\
We next claim that there exists at least one optimal solution $(\alpha_\T, p_\T)\in \mathbb A[\T]$ such that $\alpha_\T\leq \alpha_U<+\infty$, where $\alpha_U$ is obtained from Proposition \ref{prop:last_word_go_upper}. Suppose for all $(\alpha_\T, p_\T)\in \mathbb A[\T]$ such that $\alpha_U<\alpha_\T\leq +\infty$. Then, take arbitrary $(\alpha_0, p_0)\in\mathbb A[\T]$, \eqref{alpha_r_result} implies that 
\be
\CC(\alpha, p_0)=\CC(\alpha_0, p_0)=m_\T,\text{ for all }\alpha\geq \alpha_U/2.
\ee
In another word, we have $(\alpha_U/2, p_0)\in \mathbb A[\T]$ as desired.\\\\
Therefore, we conclude that there exists at least one $(\alpha_\T, p_\T)\in \mathbb A[\T]$ such that 
\be
0< \alpha_\T<\alpha_U<+\infty,
\ee
which completes the proof of Theorem \ref{main_scheme_K_result}.
\end{proof}
\subsection{Extension of $\ell^p$-anisotropic total variation via Finsler metrics}\label{counterexamples}
We can further extend the $\ell^p$-(an)-isotropic total variational by using the Finsler metric (see \cite{MR2067663} and Definition \ref{MR2067663}). Let $\abs{\cdot}_\om$: $\rn\to[0,+\infty)$ be a \emph{Finsler metric}. That is, we assume that the function $\abs{\cdot}$ is convex and satisfies the properties
\be
\abs{x}_\om\geq C\abs{x}_2,\,\, \abs{ax}_\om=a\abs{x}_\om,\,\, x\in\rn,\,\,a\geq 0, 
\ee
where $C\in\R^+$ is a positive constant. 
Then, we define the $\om$-total variation by
\be\label{TV_fisher}
TV_\om(u):=\sup\flp{\int_Q u\,\divg\vp\, dx:\,\,\vp\in C_c^\infty(Q;\,\rn),\,\,\abs{\vp}_\om^\ast\leq 1}<+\infty.
\ee
\begin{define}\label{MR2067663}
We say a collection $\mathbb F$ of Finsler metrics is training compatible if the following assertions hold.
\begin{enumerate}[1.]
\item
For any $\om\in\mathbb F$, $\abs{\cdot}_\om$: $\rn\to[0,+\infty)$ is a convex, positively 1-homogeneous function, and $\abs{x}_\om>0$ if $x\neq 0$.
\item
We denote the unit sphere of $\om$ by 
\be
S_\om:=\flp{x\in\rn:\,\, \abs{x}_{\om}^\ast=1}.
\ee
Then, we say $\om_n\wtof \om$ in $\mathbb F$ if 
\be
\dist(S_{\om_n}, S_\om)\to 0.
\ee
\item
(compactness) For any sequence $\seqn{\om_n}\subset\mathbb F$, there exists a subsequence, still denote by $\om_n$, such that $\om_n\wtof\om$ in $\mathbb F$.
\end{enumerate}
\end{define}
We present a similar version of Theorem \ref{main_scheme_K_result} but with $TV_\om$ variation. First, we introduce the training scheme $(\mathcal T_{\mathbb F})$ by
\begin{flalign}
\text{Level 1. }&\,\,\,\,\,\,\,\,\,\,\,\,\,\,\,\,\,\,\,\,\,\,( \alpha_\T, \om_\T)\in\mathbb A[\T_{\mathbb F}]:=\argmin\flp{\norm{u_{\alpha,\om}-u_c}_{L^2(Q)}^2:\,\, (\alpha,\om)\in\T_{\mathbb F}},\tag{{$\mathcal T_{\mathbb F}$-L1}}\label{scheme_B1_BV_p_main2}&\\
\text{Level 2. }&\,\,\,\,\,\,\,\,\,\,\,\,\,\,\,\,\,\,\,\,\,\,u_{\alpha,\om}:=\argmin\flp{\norm{u-u_\eta}_{L^2(Q)}^2+\alpha {TV_\om(u)}:\,\, u\in BV_\om(Q)}\tag{{$\mathcal T_{\mathbb F}$-L2}}\label{scheme_B2_BV_p_main2}.&
\end{flalign}
with the training ground
\be
\T_{\mathbb F}:=\R^+\times \mathbb F.
\ee
\begin{theorem}[Existence of solutions of scheme $(\mathcal T_{\mathbb F})$]\label{main_scheme_K_result2}
Let $u_c$ and $u_\eta\in BV(Q)$ be given such that 
\be
\inf\flp{TV_\om(u_\eta):\,\,\om\in\mathcal W}>\sup\flp{TV_\om(u_c):\,\,\om\in\mathbb F}.
\ee
Then, the training scheme $\mathcal T_{\mathbb F}$ \eqref{scheme_B1_BV_p_main2}-\eqref{scheme_B2_BV_p_main2} admits at least one pair of solution $(\alpha_{\T_{\mathbb F}}, \om_{\T_{\mathbb F}})\in {\T_{\mathbb F}}$, provided that $\mathbb F$ is training compatible.
\end{theorem}
\begin{proof}
The proof can be obtained by following line by line of the argument presented in the proof of Theorem \ref{main_scheme_K_result}. In particular, the equivalent property \eqref{equivalent_semi_TV_p}, which used extensively in the proof of Theorem \ref{thm:new-Gamma}, can be replaced by assertions 2 and 3 in Definition \ref{MR2067663}.
\end{proof}
We conclude this section by presenting several examples of training compatible $\mathbb F$.
\begin{enumerate}[1.]
\item
The $\ell^p$ - Euclidean norm defined in \eqref{p_euclidean_def}. That is, we define
\be
\mathbb F:=\flp{\abs{x}_p:=\fsp{\sum_{i=1}^N\abs{x_i}^p}^{1/p}:\,\, p\in[1,+\infty]}.
\ee
\item
The skewed $\ell^p$ - Euclidean norm. 
\be
\mathbb F:=\flp{\abs{x}_{ap}:=\fsp{\sum_{i=1}^Na_i\abs{x_i}^p}^{1/p}:\,\, p\in[1,+\infty],\,\,\sum_{i=1}^Na_i=1,\,\,a_i>0}.
\ee
\end{enumerate}

\section{Learning of acceptable optimal solutions}\label{thaos_sec}
\subsection{Non-convexity of the assessment function and counterexamples}\label{sce_Counterexamples}
We present an explicit counterexample in one dimension ($N=1$) to show that the assessment function $\CC(\alpha,p)$ is not quasi-convex. Note that as $N=1$, we have $\abs{x}_p=\abs{x}_2$ for all $p\in[1,+\infty]$. Thus, we only need to consider the case in which $p=2$ and we abbreviate $\CC(\alpha,p)$ by $\CC(\alpha)$ in Section \ref{sce_Counterexamples}.\\\\
We define the corrupted signal $u_\eta$ (red line in Figure \ref{fig:step_conter_start}) and the clean signal $u_c$ to be (blue line in Figure \ref{fig:step_conter_start})
\be
u_\eta(x):=
\begin{cases}
-10&\text{ if }x\in (0,1/4)\\
2 & \text{ if }x\in (1/4,1/2)\\
98 &\text{ if }x\in(1/2,3/4)\\
110 & \text{ if }x\in(3/4,1),
\end{cases}
\,\,\text{ and }\,\,
u_c(x):=
\begin{cases}
0&\text{ if }x\in (0,1/4)\\
20 & \text{ if }x\in (1/4,1/2)\\
80 &\text{ if }x\in(1/2,3/4)\\
100 & \text{ if }x\in(3/4,1).
\end{cases}
\ee

\begin{figure}[!h]
\begin{subfigure}{.495\textwidth}
  \centering
        \includegraphics[width=1.0\linewidth]{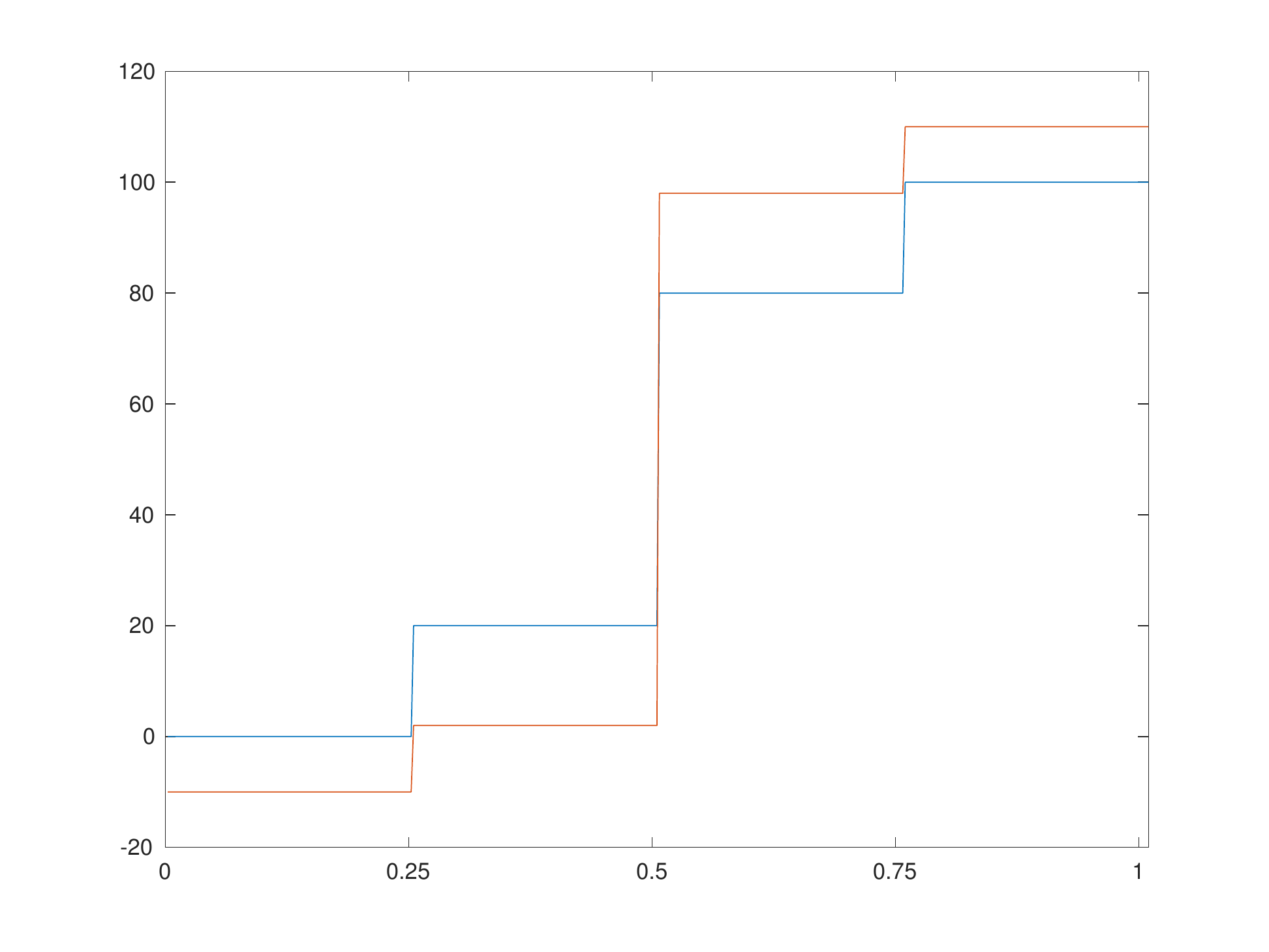}
        \caption{$u_c$ in \textcolor{blue}{blue} and $u_\eta$ in \textcolor{red}{red}}
        \label{fig:step_conter_start}
\end{subfigure}
\begin{subfigure}{.495\textwidth}
  \centering
        \includegraphics[width=1.0\linewidth]{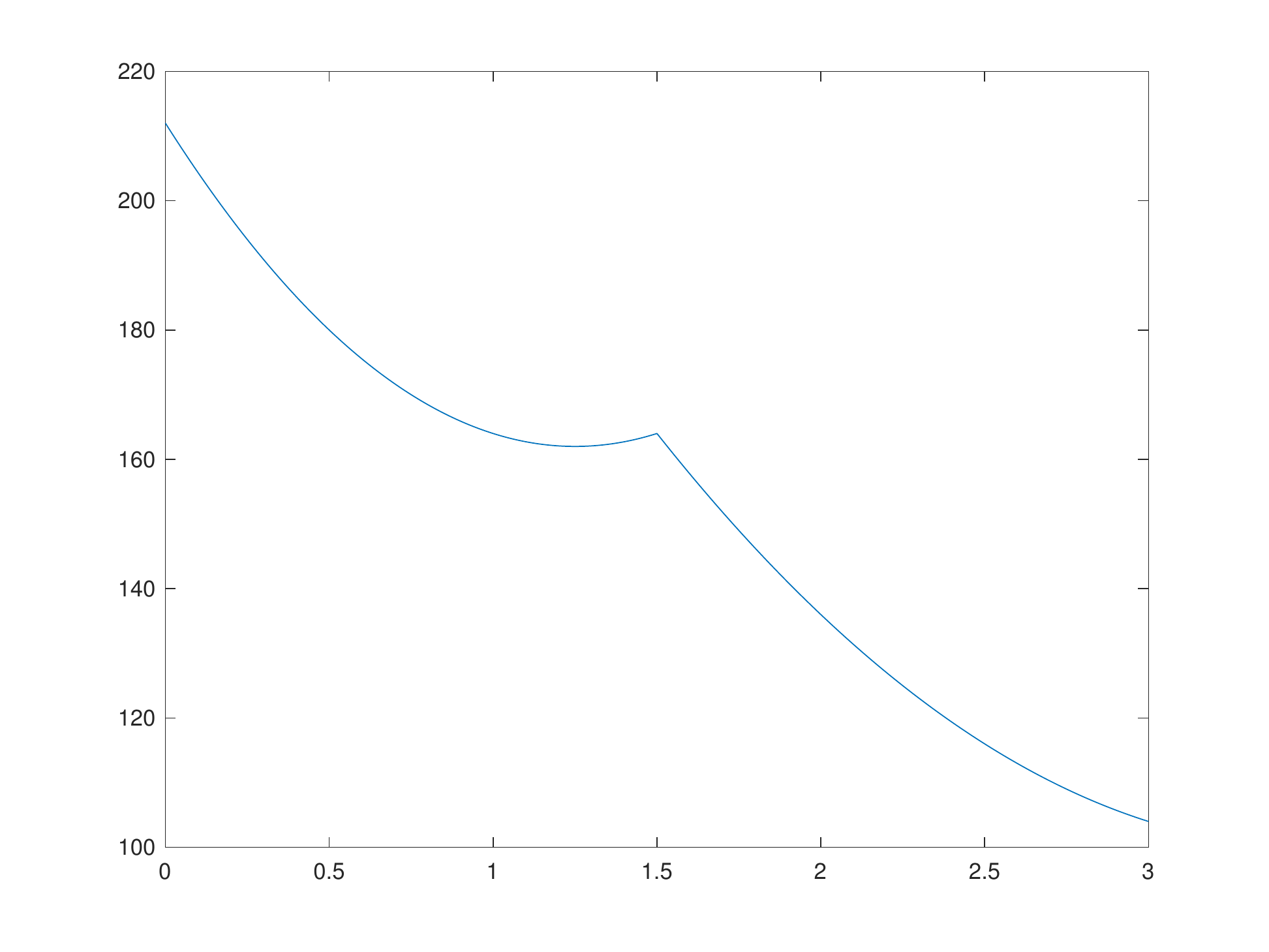}
        \caption{$\mathcal A(\alpha)$ is not quasi-convex}
        \label{fig:big_side_noise_low_resu_error}
\end{subfigure}
\caption{Figure \ref{fig:big_side_noise_low_resu_error} shows that $\mathcal A(\alpha)$ is not convex at $\alpha=1.5$.}
\label{fig:2nd_counter_exam}
\end{figure}
By \cite{strong1996exact} we can explicitly compute that  
\be
u_\alpha(x)=
\begin{cases}
-10+8\alpha&\text{ if }x\in (0,1/4)\\
2 & \text{ if }x\in (1/4,1/2)\\
98 &\text{ if }x\in(1/2,3/4)\\
110-8\alpha & \text{ if }x\in(3/4,1)
\end{cases},
\ee
for $0\leq \alpha\leq 3/2$, and consequently
\be
\mathcal A(\alpha)=\frac14\abs{8\alpha - 10}^2 + \frac14\abs{110-8\alpha-100}^2+18^2/2=\frac14\abs{8\alpha-10}^2+\frac14\abs{10-8\alpha}^2+18^2/2,
\ee
for $0\leq \alpha\leq 3/2$. Hence, we have 
\be
\mathcal A(\alpha)':=\frac d{d\alpha}\CC(\alpha)=8(8\alpha-10)\text{ and }\mathcal A(\alpha)''=64>0
\ee
for $0\leq \alpha\leq 3/2$. That is, $\mathcal A(\alpha)$ is convex and {decreasing} for $0\leq \alpha\leq 5/4$ and {increasing} for $5/4\leq \alpha\leq 3/2$ (see the first convex part in Figure \ref{fig:big_side_noise_low_resu_error}).\\\\
Next, by applying \cite{strong1996exact} again, we have, for $3/2\leq \alpha\leq 12$, that
\be
u_{\alpha}(x)=
\begin{cases}
2+4(\alpha-1.5)&\text{ if }x\in (0,1/2)\\
98-4(\alpha-1.5) & \text{ if }x\in(1/2,1).
\end{cases}
\ee
Hence, we have, for $3/2\leq \alpha\leq 12$,
\begin{align}
\mathcal A(\alpha)
&=\frac14\abs{2+4(\alpha-1.5)}^2+\frac14\abs{20-(2+4(\alpha-1.5))}^2 \\
&+ \frac14\abs{80-(98-4(\alpha-1.5))}^2 + \frac14\abs{100-(98-4(\alpha-1.5))}^2
\end{align}
which implies
\be
\mathcal A(\alpha)'=8\alpha-20\text{ and }\mathcal A(\alpha)''=8>0.
\ee
{Therefore, we see that $\mathcal A(\alpha)'<0$ for $5/2<\alpha<2$, i.e., $\mathcal A(\alpha)$ is again decreasing (this is the third convex part in Figure \ref{fig:big_side_noise_low_resu_error}), and hence $\mathcal A(\alpha)$ is not quasi-convex.
\subsection{A finite approximation of scheme $\mathcal T$}\label{thaos_sec_finite}
In Section \ref{thaos_sec_finite} we assume that $u_\eta\in BV(Q)$. We introduce first the concept of (finite) \emph{Training Ground}.
\begin{define}\label{finite_playground_l}
Let $l\in\N$ and recall the upper bound $\alpha_U\in\R^+$ from Proposition \ref{prop:last_word_go_upper}.
\begin{enumerate}[1.]
\item
By Theorem \ref{main_scheme_K_result}, we can reduce the training ground $\T$ to 
\be
\mathbb T:=[0,\alpha_U]\times[1,+\infty].
\ee
\item
We define $\delta_l:=1/l$ and we write
\be\label{finite_grid_alpha}
T_l[\alpha]:=\bigcup_{k=1}^l\tilde T_k[\alpha],
\ee
where
\be
\tilde T_k[\alpha]:=\flp{0,\,\,\delta_k,\,\,2\delta_k,\,\,\ldots,\,\,i\delta_k,\ldots, \alpha_U}.
\ee
Similarly, we denote by $T_l[p]$ that 
\be\label{finite_grid_p}
T_l[p]:=\bigcup_{k=1}^l{\tilde T_k[p]},
\ee
where
\be
\tilde T_k[p]:=\flp{+\infty}\cup\flp{1/\delta_k,\,\,1/(2\delta_k),\,\,\ldots,\,\,1/(i\delta_k),\,\,\ldots,1}.
\ee
\item
We define the \emph{Finite Training Ground} $\T_l$ at step $l\in\N$ by
\be
\T_l:=T_l[\alpha]\times T_l[p].
\ee
\item
For $i$, $j\in\N$ and $1\leq i,j\leq l$, we define the $(i,j)$-th \emph{finite grid} $\mathbb G_l(i,j)$ by 
\be\label{eq_finite_grid_one}
\mathbb G_l(i,j):=[i\delta_l,(i+1)\delta_l]\times[j(\delta_l^{-1}),(j+1)(\delta_l^{-1})].
\ee
\end{enumerate}
\end{define}
\begin{remark}\label{rmk_nested_tg_fg}
We draw the following observations from Definition \ref{finite_playground_l}.
\begin{enumerate}[1.]
\item
$\#\flp{\T_l}<+\infty$ for each $l\in\N$;
\item
$\Tl\subset \T_{l+1}\subset\cdots\subset \T$;
\item
We have
\be\label{nested_use_to}
\T\subset \operatorname{cl}\fsp{\bigcup_{l=1}^{+\infty}\T_l}.
\ee
\item
$\mathbb G_l(i,j)\subset \T\subset \operatorname{cl}(\R^2)$  is a closed set with positive $\mathcal L^2$ measure.
\end{enumerate}
\end{remark}
The training scheme $\mathcal T$ with finite training ground $\Tl$ can be presented as follows:
\begin{flalign}
\text{Level 1. }&\,\,\,\,\,\,\,\,\,\,\,\,\,\,\,\,\,\,\,\,( \alpha_\Tl, p_\Tl)\in\mathbb A[\Tl]:=\argmin\flp{\norm{u_{\alpha,p}-u_c}_{L^2(Q)}^2:\,\, (\alpha,p)\in \T_l},\tag{{$\mathcal T_l$-L1}}\label{scheme_B1_BV_p_finite}&\\
\text{Level 2. }&\,\,\,\,\,\,\,\,\,\,\,\,\,\,\,\,\,\,\,\,u_{\alpha,p}:=\argmin\flp{\norm{u-u_\eta}_{L^2(Q)}^2+\alpha {TV_p(u)}:\,\, u\in BV(Q)}\tag{{$\mathcal T_l$-L2}}\label{scheme_B2_BV_p_finite}.&
\end{flalign}
\begin{theorem}\label{article_result}
Let $u_c$ and $u_\eta\in BV(Q)$ be given such that \eqref{req_main_scheme_K_result} is satisfied, and let $\alpha_U$ be obtained from Proposition \ref{prop:last_word_go_upper}. Then the following assertions hold.
\begin{enumerate}[1.]
\item
As $l\to\infty$, we have that 
\be
\operatorname{dist}(\mathbb A[\T],\mathbb A[\Tl])\to 0.
\ee
\item
For each $l\in\N$, there holds
\be\label{intro_estimation_last}
\abs{\CC(\alpha_\Tl, p_\Tl)-\CC( \alpha_\T, p_\T)}\leq  \sqrt{\alpha_U}\fmp{{1/l}^{1/2}+2{\fsp{1-N^{-1/\sqrt l}}}^{1/2}}[TV_1(u_\eta)]^{1/2}.
\ee
\end{enumerate}
\end{theorem}
This theorem will be proved in several propositions.
\begin{proposition}\label{gammaT}
Recall $\Tl$ from Definition \ref{finite_playground_l} and the optimal set $\mathbb A[\Tl]$ defined by
\be
\mathbb A[\Tl]=\argmin\flp{\mathcal A(\alpha,p):\,\,(\alpha,p)\in \T_l},\text{ for each }l\in\N.
\ee
Then, all cluster points of sequence of sets $\seql{\mathbb A[\Tl]}$ belongs to the collection $\mathbb A[\T]$.
\end{proposition}
\begin{proof}
We claim first that 
\be\label{finite_grid_min_conv}
\min\flp{\mathcal A(\alpha,p):\,\,(\alpha,p)\in \T_l}:=m_l\to m:=\min\flp{\mathcal A(\alpha,p):\,\,(\alpha,p)\in\T}.
\ee
Not that, for each $l\in\N$ that $\Tl\subset \T_{l+1}$ by Remark \ref{rmk_nested_tg_fg}. Thus, we have
\be
m_l\geq m_{l+1}\geq \cdots\geq 0.
\ee
In view of Monotone Convergence Theorem, with no further subsequence extracted, there exists $\bar m\geq 0$ such that  
\be\label{monotone_down}
m_l\searrow \bar m.
\ee
We show that $\bar m=m$. Suppose $\bar m>m$. Then there exists $\bar l\in \N$ such that 
\be\label{eq_greater_ml}
 m_l>m+\frac12(\bar m-m)\text{, for all }l\geq\bar l.
\ee
On the other hand, by \eqref{nested_use_to}, for any $(\alpha_\T, p_\T)\in\mathbb A[\T]$ we can extract a sequence $\seql{(\alpha_l,p_l)}\subset\T$, in which $(\alpha_l,p_l)\in \T_l$ for each $l\in\N$, such that $(\alpha_l,p_l)\to(\alpha_\T, p_\T)$. Thus, by \eqref{finite_grid_min_conv} we have
\be\label{eq_greater_ml3}
\CC(\alpha_l,p_l)\geq m_l.
\ee
Next, by Theorem \ref{thm:new-Gamma} we have 
\bes
\CC(\alpha_l,p_l)\to\CC(\alpha_\T, p_\T)=m,
\ees
which implies that there exists $\tilde l\in\N$ such that 
\be\label{eq_greater_ml2}
\CC(\alpha_l,p_l)\leq m+\frac14(\bar m-m),\text{ for all }l\geq \tilde l.
\ee
Hence, by \eqref{eq_greater_ml}, \eqref{eq_greater_ml3}, and \eqref{eq_greater_ml2}, we must have
\be
m_l\leq m+\frac14(\bar m-m)<m+\frac12(\bar m-m)<m_l,
\ee
which is a contradiction.\\\\
To finish, we point out that all cluster points of the sequence $\seql{\mathbb A[\Tl]}$ satisfy \eqref{finite_grid_min_conv} since there is no subsequence extracted from \eqref{monotone_down} due to the property of the monotone convergence theorem, and hence we conclude this proposition.
\end{proof}
\begin{proposition}\label{estimation_iso_like}
Let $\alpha>0$, $\e>0$, $p\in[1,+\infty]$. Then we have 
\be
\norm{u_{\alpha+\e,p}-u_{\alpha,p}}_{L^2(Q)}^2\leq \e TV_p(u_{\alpha,p}).
\ee
\end{proposition}
\begin{proof}
By the minimality of $u_{\alpha,p}$ there holds
\be
u_{\alpha,p}-u_\eta = -\alpha \partial TV_p(u_{\alpha,p})
\ee
and
\be
u_{\alpha+\e,p}-u_\eta = -(\alpha+\e) \partial TV_p(u_{\alpha+\e,p}).
\ee
Subtracting one from another, we have that
\be
u_{\alpha,p}-u_{\alpha+\e,p}+\alpha\fmp{\partial TV_p(u_{\alpha,p})-\partial TV_p(u_{\alpha+\e,p})}-\e \partial TV_p(u_{\alpha+\e,p})=0.
\ee
Next, by multiplying with $u_{\alpha,p}-u_{\alpha+\e,p}$ and integrating over $Q$, we deduce that 
\begin{align}\label{est_iso_eq}
\norm{u_{\alpha+\e,p}-u_{\alpha,p}}_{L^2(Q)}^2&+\alpha\fjp{u_{\alpha,p}-u_{\alpha+\e,p},\partial TV_p(u_{\alpha,p})-\partial TV_p(u_{\alpha+\e,p})}_{L^2}\\
&-\e\fjp{\partial TV_p(u_{\alpha+\e,p}),u_{\alpha,p}-u_{\alpha+\e,p}}_{L^2}=0.
\end{align}
Since $\partial TV_p$ is a maximal monotone operator (see \cite[Proposition 5.5, page 25]{MR1727362}), we obtain that
\be
\fjp{u_{\alpha,p}-u_{\alpha+\e,p},\partial TV_p(u_{\alpha,p})-\partial TV_p(u_{\alpha+\e,p})}_{L^2}\geq 0.
\ee
This, and together with \eqref{est_iso_eq}, we have that 
\begin{align}
&\norm{u_{\alpha+\e,p}-u_{\alpha,p}}_{L^2(Q)}^2\\
&\leq \e\fjp{\partial TV_p(u_{\alpha+\e,p}),u_{\alpha,p}-u_{\alpha+\e,p}}_{L^2}\leq \e\fmp{TV(u_{\alpha,p})-TV(u_{\alpha+\e,p})},
\end{align}
where at the last inequality we used the property of sub-gradient (see \cite[Proposition 5.4, page 24]{MR1727362}). In turn, we obtain that 
\be
\norm{u_{\alpha+\e,p}-u_{\alpha,p}}^2_{L^2(Q)}\leq \e TV(u_{\alpha,p}),
\ee
which concludes the proof.
\end{proof}
%
%
%
%
\begin{proposition}\label{estimation_iso_aniso}
Let $p\geq1$, $\e>0$, $\alpha\in\R^+$. Then we have 
\be
\norm{u_{\alpha,p+\e}-u_{\alpha,p}}^2_{L^2(Q)}\leq \alpha\fmp{1-N^{1/(p+\e)-1/p}} \fmp{TV_{p+\e}(u_{\alpha,p})+TV_{p+\e}(u_{\alpha,p+\e})}.
\ee
\end{proposition}
\begin{proof}
Instead of using the sub-gradient operator as used in Proposition \ref{estimation_iso_like}, we proceed with the first variation of \eqref{scheme_B2_BV_p_finite}. Since in this argument $\alpha\in\R^+$ is fixed, we abbreviate $u_{\alpha,p}$ by $u_p$.\\
As first suggested in \cite{rudin1992nonlinear}, we can regularize the $TV_p$ seminorm by a factor $\delta>0$ and consider
\be\label{updelta_minarg}
u_p^\delta:=\argmin\flp{\norm{u-u_\eta}_{L^2(Q)}^2+\alpha TV_p^\delta(u):\,\, u\in BV(Q)}
\ee
where
\be
TV_p^\delta(u):=\int_Q\sqrt{\abs{\nabla u}^2_{\ell^p}+\delta}\,dx.
\ee
The new $TV_p^\delta$ seminorm is differentiable in $u$. For arbitrary $v\in BV(Q)$, we write the first variation of \eqref{updelta_minarg} as follows
\be
\int_Q\fsp{u_{p+\e}^\delta-u_\eta}v\,dx +\alpha \int_Q \frac{\nabla u_{p+\e}^\delta}{\sqrt{\abs{\nabla u_{p+\e}^\delta}^2_{\ell^{p+\e}}+\delta}}\nabla v\,dx=0
\ee
and
\be
\int_Q\fsp{u_{p}^\delta-u_\eta}v\,dx + \alpha\int_Q \frac{\nabla u_{p}^\delta}{\sqrt{\abs{\nabla u_{p}^\delta}^2_{\ell^p}+\delta}}\nabla v\,dx=0.
\ee
Subtracting one from another, we have that 
\begin{align}\label{p_var_est_1}
\frac1\alpha\int_Q \fsp{u_{p+\e}^\delta-u_{p}^\delta}v\,dx&=\int_Q\fmp{ \frac{\nabla u_{p}^\delta}{\sqrt{\abs{\nabla u_{p}^\delta}^2_{\ell^p}+\delta}}- \frac{\nabla u_{p+\e}^\delta}{\sqrt{\abs{\nabla u_{p+\e}^\delta}^2_{\ell^{p+\e}}+\delta}}}\nabla v\,dx\\
&=\int_Q\fmp{ \frac{\nabla u_{p}^\delta}{\sqrt{\abs{\nabla u_{p}^\delta}^2_{\ell^p}+\delta}}- \frac{\nabla u_{p+\e}^\delta}{\sqrt{\abs{\nabla u_{p+\e}^\delta}^2_{\ell^p}+\delta}}}\nabla v\,dx\\
&\,\,\,\,\,\,\,\,\,+\int_Q\fmp{ \frac{\nabla u_{p+\e}^\delta}{\sqrt{\abs{\nabla u_{p+\e}^\delta}^2_{\ell^p}+\delta}}- \frac{\nabla u_{p+\e}^\delta}{\sqrt{\abs{\nabla u_{p+\e}^\delta}^2_{\ell^{p+\e}}+\delta}}}\nabla v\,dx.
\end{align}
Set $v:=u_{p+\e}^\delta-u_{p}^\delta$. We compute that
\begin{align}\label{p_var_est_2}
&\int_Q\fmp{ \frac{\nabla u_{p}^\delta}{\sqrt{\abs{\nabla u_{p}^\delta}^2_{\ell^p}+\delta}}- \frac{\nabla u_{p+\e}^\delta}{\sqrt{\abs{\nabla u_{p+\e}^\delta}^2_{\ell^p}+\delta}}}\nabla (u_{p+\e}^\delta-u_{p}^\delta)\,dx\\
& = \int_Q\fmp{ -\divg\frac{\nabla u_{p}^\delta}{\sqrt{\abs{\nabla u_{p}^\delta}^2_{\ell^p}+\delta}}+\divg \frac{\nabla u_{p+\e}^\delta}{\sqrt{\abs{\nabla u_{p+\e}^\delta}^2_{\ell^p}+\delta}}} (u_{p+\e}^\delta-u_{p}^\delta)\,dx\\
&\leq 0,
\end{align}
where at the last inequality we used the fact that 
\be
 -\divg\frac{\nabla u}{\sqrt{\abs{\nabla u}^2_{\ell^p}+\delta}}\in\partial TV^\delta_p(u)
\ee
is a maximal monotone operator.\\
Next, we compute that 
\begin{align}\label{p_var_est_add1}
&\int_Q\fmp{ \frac{\nabla u_{p+\e}^\delta}{\sqrt{\abs{\nabla u_{p+\e}^\delta}^2_{\ell^p}+\delta}}- \frac{\nabla u_{p+\e}^\delta}{\sqrt{\abs{\nabla u_{p+\e}^\delta}^2_{\ell^{p+\e}}+\delta}}}\nabla u_{p+\e}^\delta \,dx\\
&=\int_Q\fmp{ \frac{1}{\sqrt{\abs{\nabla u_{p+\e}^\delta}^2_{\ell^p}+\delta}}- \frac{1}{\sqrt{\abs{\nabla u_{p+\e}^\delta}^2_{\ell^{p+\e}}+\delta}}}\fsp{\nabla u_{p+\e}^\delta\cdot\nabla u_{p+\e}^\delta} dx\\
&\leq \int_Q \abs{\frac{1}{\sqrt{\abs{\nabla u_{p+\e}^\delta}^2_{\ell^p}+\delta}}-\frac{1}{\sqrt{\abs{\nabla u_{p+\e}^\delta}^2_{\ell^{p+\e}}+\delta}}}\abs{\nabla u_{p+\e}^\delta\cdot\nabla u_{p+\e}^\delta} \,dx\\
&\leq  \fmp{1-\frac{1}{N^{1/p-1/{p+\e}}}}\int_Q \abs{\frac{1}{\sqrt{\abs{\nabla u_{p+\e}^\delta}^2_{\ell^p}+\delta}}}\abs{\nabla u_{p+\e}^\delta}_{\ell^{p+\e}}\abs{\nabla u_{p+\e}^\delta}^\ast_{\ell^{p+\e}} \,dx\\
&\leq \fmp{1-\frac{1}{N^{1/p-1/{p+\e}}}} TV_{p+\e}(u_{p+\e}^\delta),
\end{align}
where at the second inequality we used \eqref{equivalent_p_norm} and H\"older inequality. We could similarly estimate that 
\begin{align}
&\int_Q \fmp{\frac{\nabla u_{p+\e}^\delta}{\sqrt{\abs{\nabla u_{p+\e}^\delta}^2_{\ell^p}+\delta}}-\frac{\nabla u_{p+\e}^\delta}{\sqrt{\abs{\nabla u_{p+\e}^\delta}^2_{\ell^{p+\e}}+\delta}}}(-\nabla u_{p}^\delta) \,dx\\
&\leq \fmp{1-\frac{1}{N^{1/p-1/{p+\e}}}}\int_Q \abs{\frac{1}{\sqrt{\abs{\nabla u_{p+\e}^\delta}^2_{\ell^{p+\e}}+\delta}}}\abs{\nabla u_{p+\e}^\delta}_{\ell^{p+\e}}\abs{\nabla u_{p}^\delta}_{l^\ast_{p+\e}} \,dx\\
&\leq \fmp{1-\frac{1}{N^{1/p-1/{p+\e}}}} TV_{p+\e}(u_{p}^\delta).
\end{align}
This, and together with \eqref{p_var_est_add1}, we conclude that 
\begin{align}\label{p_var_est_3}
&\int_Q\fmp{\frac{\nabla u_{p+\e}^\delta}{\sqrt{\abs{\nabla u_{p+\e}^\delta}^2_{\ell^p}+\delta}}-\frac{\nabla u_{p+\e}^\delta}{\sqrt{\abs{\nabla u_{p+\e}^\delta}^2_{\ell^{p+\e}}+\delta}}}\nabla (u_{p+\e}^\delta-u_{p}^\delta)\,dx\\
&\leq  \fmp{1-N^{1/(p+\e)-1/p}} \fmp{TV_{p+\e}(u_{p}^\delta)+TV_{p+\e}(u_{p+\e}^\delta)}.
\end{align}
Hence, by \eqref{p_var_est_1}, \eqref{p_var_est_2}, and \eqref{p_var_est_3}, we have
\be
\norm{u_p^\delta-u^\delta_{p+\e}}_{L^2(Q)}^2\leq \alpha\fmp{1-N^{1/(p+\e)-1/p}} \fmp{TV_{p+\e}(u^\delta_p)+TV_{p+\e}(u^\delta_{p+\e})}.
\ee
Moreover, since $u_p^\delta\to u_p$ in the strict topology of $BV$ (see \cite{rudin1992nonlinear}), it follows that
\begin{align}
\norm{u_p-u_{p+\e}}_{L^2(Q)}^2&=\lim_{\delta\to 0}\norm{u_p^\delta-u^\delta_{p+\e}}_{L^2(Q)}^2\\
&\leq \alpha\fmp{1-N^{1/(p+\e)-1/p}}\limsup_{\delta\to 0} \fmp{TV_{p+\e}(u^\delta_p)+TV_{p+\e}(u^\delta_{p+\e})}\\
&=  \alpha\fmp{1-N^{1/(p+\e)-1/p}} \fmp{TV_{p+\e}(u_p)+TV_{p+\e}(u_{p+\e})},
\end{align}
which completes the proof of the proposition.
\end{proof}
We recall the reduced training ground $\T$ from Definition \ref{finite_playground_l}.
\begin{corollary}\label{cor_ready_for_estimation}
Let $(\alpha_1,p_1)$, $(\alpha_2,p_2)\in\mathbb T$. Then we have
\be
\norm{u_{\alpha_2,p_2}-u_{\alpha_1,p_1}}_{L^2(Q)}\leq \fmp{\abs{\alpha_1-\alpha_2}^{1/2}+2\fsp{\alpha_U\fsp{1-N^{-\abs{1/p_2-1/p_1}}}}^{1/2}}\fmp{TV_1(u_\eta)}^{1/2}.
\ee
\end{corollary}
\begin{proof}
By Proposition \ref{estimation_iso_like} and Proposition \ref{estimation_iso_aniso}, we observe that
\begin{align}
&\norm{u_{\alpha_2,p_2}-u_{\alpha_1,p_1}}_{L^2(Q)}\\
&\leq \norm{u_{\alpha_2,p_2}-u_{\alpha_1,p_2}}_{L^2(Q)}+\norm{u_{\alpha_1,p_2}-u_{\alpha_1,p_1}}_{L^2(Q)}\\
&\leq \fmp{\abs{\alpha_1-\alpha_2} TV_{p_2}(u_{\alpha_1,p_2})}^{1/2}+\fmp{\alpha_1\fsp{1-N^{-\abs{1/p_2-1/p_1}}} \fsp{TV_{p_2}(u_{\alpha_1,p_1})+TV_{p_2}(u_{\alpha_1,p_2})}}^{1/2}.
\end{align}
Next, in view of Lemma \eqref{b_differential_aa}, we have $TV_p(u_{\alpha,p})\leq TV_p(u_\eta)\leq TV_1(u_\eta)$ and hence we conclude that
\be
\norm{u_{\alpha_2,p_2}-u_{\alpha_1,p_1}}_{L^2(Q)}\leq \fmp{\abs{\alpha_1-\alpha_2}^{1/2}+2\fsp{\alpha_U\fsp{1-N^{-\abs{1/p_2-1/p_1}}}}^{1/2}}\fmp{TV_1(u_\eta)}^{1/2}
\ee
as desired.
\end{proof}
We are now ready to proof Theorem \ref{article_result}.
\begin{proof}[Proof of Theorem \ref{article_result}]
The Assertion 1 can be deduced from Proposition \ref{gammaT} directly. \\\\
We next prove Assertion 2. Let $(\alpha_1,p_1)$, $(\alpha_2,p_2)\in\mathbb T$. By Corollary \ref{cor_ready_for_estimation} we have that 
\begin{align}\label{two_step_estimation}
&\abs{\CC(\alpha_1, {p}_1)- \CC(\alpha_2, p_2)}\\
&= \abs{\norm{u_{\alpha_1, {p}_1}-u_c}_{L^2(Q)}-\norm{u_{\alpha_2,{p_2}}-u_c}_{L^2(Q)}}
\leq \norm{u_{\alpha_1, {p}_1}-u_{\alpha_2,{p_2}}}_{L^2(Q)}\\
&\leq \fmp{\abs{\alpha_1-\alpha_2}^{1/2}+2\fsp{\alpha_U\fsp{1-N^{-\abs{1/p_2-1/p_1}}}}^{1/2}}\fmp{TV_1(u_\eta)}^{1/2}.
\end{align}
Let $(\alpha_\T,p_\T)\in\mathbb A[\T]$ and a minimizing sequence $(\alpha_\Tl,p_\Tl)\in\mathbb A[\Tl]$  
 such that $(\alpha_\Tl, {p}_\Tl)\to (\alpha_\T, {p}_\T)$ as $\l\to\infty$. Also, for each $l\in\N$, we fix a finite grid $\mathbb G_l(i_l,j_l)$ (recall  \eqref{eq_finite_grid_one}) such that 
\be\label{where_the_true_opt}
(\alpha_\T, p_\T)\in \mathbb G_l(i_l,j_l)\subset\R^2.
\ee
Thus, since $\mathbb G_l(i_l,j_l)$ is closed, we have
\be\label{where_the_true_opt2}
\CC(\alpha_\T, p_\T)= \min\flp{\CC(\alpha,{p}):\,\,(\alpha,{p})\in \mathbb G_l(i_l,j_l)}.
\ee
Also, in view of \eqref{two_step_estimation} and Theorem \ref{thm:new-Gamma}, there holds
\begin{align}\label{grid_estimation_one}
&\max\flp{\CC(\alpha,{p}):\,\,(\alpha,{p})\in \mathbb G_l(i_l,j_l)} -\min\flp{\CC(\alpha,{p}):\,\,(\alpha,{p})\in \mathbb G_l(i_l,j_l)}\\
&\leq   \fmp{\abs{\delta_l}^{1/2}+2\fsp{\alpha_U\fsp{1-N^{-\delta_l}}}^{1/2}}\fmp{TV_1(u_\eta)}^{1/2}.
\end{align}
Next, if at step $l$ that $(\alpha_\Tl,{p}_\Tl)\in \mathbb G_l(i_l,j_l)$, we can immediately deduce that
\begin{align}\label{shangxia_1}
&\CC(\alpha_\Tl,  {p}_\Tl)-\CC(\alpha_\T,{p_\T})\\
&\leq \max\flp{\CC(\alpha,{p}):\,\,(\alpha,{p})\in \mathbb G_l(i_l,j_l)} - \CC(\alpha_\Tl,  {p}_\Tl)\\
&\leq \max\flp{\CC(\alpha,{p}):\,\,(\alpha,{p})\in \mathbb G_l(i_l,j_l)} - \min\flp{\CC(\alpha,{p}):\,\,(\alpha,{p})\in \mathbb G_l(i_l,j_l)},
\end{align}
where for the last inequality we used \eqref{where_the_true_opt2}.\\\\
If $(\alpha_\Tl,{p}_\Tl)\notin \mathbb G_l(i_l,j_l)$, then in view of the definition of $(\alpha_\Tl,{p}_\Tl)$, we must have
\be\label{anyway_minimizer}
\max\flp{\CC(\alpha,{p}):\,\,(\alpha,{p})\in \mathbb G_l(i_l,j_l)\cap\mathbb T_l}\geq \CC(\alpha_\Tl, {p}_\Tl).
\ee
Hence, by \eqref{anyway_minimizer} we again obtain that  
\begin{align}\label{shangxia_2}
&\CC(\alpha_\Tl,  {p}_\Tl)-\CC(\alpha_\T,{p}_\T)\\
&\leq\max\flp{\CC(\alpha,{p}):\,\,(\alpha,{p})\in \mathbb G_l(i_l,j_l)\cap\mathbb T_l}- \CC(\alpha_\T,  p_\T)\\
&\leq \max\flp{\CC(\alpha,{p}):\,\,(\alpha,{p})\in \mathbb G_l(i_l,j_l)} - \min\flp{\CC(\alpha,{p}):\,\,(\alpha,{p})\in \mathbb G_l(i_l,j_l)},
\end{align}
where at the last inequality we used the assumption \eqref{where_the_true_opt}. In the end, by \eqref{two_step_estimation}, \eqref{shangxia_1}, and \eqref{shangxia_2}, we observe that 
\begin{align}
&\CC(\alpha_\Tl,  p_\Tl)-\CC(\alpha_\T,  p_\T)\\
&\leq \max\flp{\CC(\alpha,{p}):\,\,(\alpha,{p})\in\partial \mathbb G_l(i_l,j_l)} -  \min\flp{\CC(\alpha,{p}):\,\,(\alpha,{p})\in \mathbb G_l(i_l,j_l)}\\
&\leq   \fmp{{\delta_l}^{1/2}+2\fsp{\alpha_U\fsp{1-N^{-\delta_l}}}^{1/2}}(TV_1(u_\eta))^{1/2}\\
&\leq  \sqrt{\alpha_U}\fmp{\abs{1/l}^{1/2}+2{\fsp{1-N^{-1/\sqrt l}}}^{1/2}}(TV_1(u_\eta))^{1/2}
\end{align}
and hence the thesis.
\end{proof}
\begin{remark}\label{rmk_why_thm_main}
We point out that the optimal solutions $(\alpha_\Tl, p_\Tl)\in\mathbb A[\T_l]$ can be determined precisely since $\#\flp{\T_l}<+\infty$ at each $l\in\N$. Then, for any given acceptable error, we can determine the required approximation step $l$ by using \eqref{intro_estimation_last}.
\end{remark}
\subsubsection{A relaxation of corrupted image $u_\eta$}\label{relaxation_corrupted_sec}
The assumption in Theorem \ref{article_result} that $u_\eta\in BV(Q)$ is a rather strong one and, in fact, not realistic for image denoising. We argue, however, that the error bound \eqref{scheme_B1_BV_p_main} can still be used in practice by replacing the noisy image $u_\eta$ with an approximation that has bounded variation. To be precise, we consider a sequence $\flp{u_\eta^K}_{K=1}^\infty\subset BV(Q)$ such that 
\be\label{smooth_approx_u_eta}
u_\eta^K\to u_\eta\text{ strongly in }L^2(Q)\text{ and }TV(u_\eta^K)<+\infty
\ee
and introduce the training scheme $\mathcal T^K$ (\eqref{L1_resolution_K}-\eqref{L2_resolution_K}) as follows.
\begin{flalign}
\text{Level 1. }&\,\,\,\,\,\,\,\,\,\,\,\,\,\,\,\,\,\,\,\,\,\,( \alpha_\T^K, p_\T^K)\in\mathbb A^K[\T]:=\argmin\flp{\norm{u_{\alpha,p}^K-u_c}_{L^2(Q)}^2:\,\, (\alpha,p)\in\T},\tag{{$\mathcal T^K$-L1}}\label{L1_resolution_K}&\\
\text{Level 2. }&\,\,\,\,\,\,\,\,\,\,\,\,\,\,\,\,\,\,\,\,\,\,u^K_{\alpha,p}:=\argmin\flp{\norm{u-u^K_\eta}_{L^2(Q)}^2+\alpha {TV_p(u)}:\,\, u\in BV(Q)}\tag{{$\mathcal T^K$-L2}}\label{L2_resolution_K}.&
\end{flalign}
We also define the assessment function with respect to $u_\eta^K$ by
\be
\CC^K(\alpha,p):=\norm{u_{\alpha,p}^K-u_c}_{L^2(Q)}.
\ee
\begin{theorem}\label{article_result_finite_resolution}
Let $K\in\N$, $u_c\in BV(Q)$ and $u_\eta\in L^2(Q)$ be given. Let $u_\eta^K$ be defined as in \eqref{smooth_approx_u_eta}. Then the following assertions hold.
\begin{enumerate}[1.]
\item
As $K\to\infty$, we have
\be
\operatorname{dist}(\mathbb A^K[\T],\mathbb A[\T])\to 0.
\ee
\item
For each $K\in\N$, there holds
\be
\abs{\CC^K(\alpha_\T^K, p_\T^K)-\CC( \alpha_\T, p_\T)}\leq  2\norm{u_{\eta}^K-u_\eta}_{L^2(Q)}.
\ee
\end{enumerate}
\end{theorem}
Before we prove Theorem \ref{article_result_finite_resolution}, we first prove an enhanced version of Proposition \ref{compact_semi_para}.
\begin{proposition}\label{approx_bilevel}
Let $\flp{u_\eta^K}_{K=1}^\infty\subset BV(Q)$ and $\flp{\alpha_K,p_K}_{K=1}^\infty\subset \T$ be such that 
\be\label{compact_semi_para_eq1}
u_\eta^K\to u_\eta\text{ strongly in }L^2
\ee
and 
\be\label{compact_semi_para_eq2}
(\alpha_K,p_K)\to (\alpha,p)\in \T.
\ee
Then we have
\be\label{multi_conv_recon3}
u^K_{\alpha_K,p_K}\wtos u_{\alpha,p}\text{ in }BV
\ee
and
\be\label{multi_conv_recon_TV}
\lim_{K\to\infty} TV_{p_K}(u^K_{\alpha_K,p_K})=TV_p(u_{\alpha,p}).
\ee
\end{proposition}
\begin{proof}
We assume first that $\alpha>0$. By \eqref{compact_semi_para_eq1} and \eqref{compact_semi_para_eq2}, there exist $N^\ast\in\N$ such that 
\be
\frac12\alpha\leq \alpha_K \leq \alpha+1\text{ and }\norm{u_\eta^K}_{L^2(Q)}^2\leq \norm{u_\eta}_{L^2(Q)}^2+1,\text{ for all }n\geq N^\ast.
\ee
It follows that 
\begin{align}
\frac12\alpha TV_2(u^K_{\alpha_K,p_K})&\leq \max\flp{1, N^{1/p_K-1/2}}\frac12\alpha TV_{p_K}(u^K_{\alpha_K,p_K})\\
&\leq{\alpha_K}TV_{p_K}(u^K_{\alpha_K,p_K})
\leq \norm{u^K_{\alpha_K,p_K}-u_\eta^K}_{L^2(Q)}^2+{\alpha_K}TV_{p_K}(u^K_{\alpha_K,p_K})\\
&\leq \norm{u_\eta^K}_{L^2(Q)}^2\leq \norm{u_\eta}_{L^2(Q)}^2+1<+\infty.
\end{align}
Thus, we have
\be
\sup\flp{\norm{u^K_{\alpha_K,p_K}}_{BV(Q)}:\,\, n\in\N}<+\infty,
\ee
and, up to a (not-relabeled) subsequence, there exists $w\in BV(Q)$ such that 
\be\label{liminf_ineq_a1}
u^K_{\alpha_K,p_K}\wtos w\text{ in $BV$ and }u^K_{\alpha_K,p_K}\to w\text{ in }L^1\text{ and }\text{ a.e..}
\ee
We claim that $w=u_{\alpha,p}$ a.e.. Indeed, since $u^K_{\alpha_K,p_K}$ is the unique minimizer of \eqref{L2_resolution_K}, we have that 
\be
\norm{u_{\alpha,p}-{u_\eta^K}}_{L^2(Q)}^2+\alpha_KTV_{p_K}(u_{\alpha,p})\geq \norm{u^K_{\alpha_K,p_K}-{u_\eta^K}}_{L^2(Q)}^2+\alpha_KTV_{p_K}(u^K_{\alpha_K,p_K}),
\ee
and hence
\begin{align}\label{multi_conv_recon1}
&\liminf_{K\to\infty}\norm{u_{\alpha,p}-{u_\eta^K}}_{L^2(Q)}^2+\alpha_KTV_{p_K}(u_{\alpha,p})\\
&\geq\liminf_{K\to\infty} \norm{u^K_{\alpha_K,p_K}-{u_\eta^K}}_{L^2(Q)}^2+\liminf_{K\to\infty}\alpha_KTV_{p_K}(u^K_{\alpha_K,p_K})\\
& \geq \norm{w-u_\eta}_{L^2(Q)}^2+\alpha TV_p(w),
\end{align}
where at the last inequality we used Fatou's lemma and \eqref{liminf_ineq_a1}. On the other hand, we have
\be\label{multi_conv_recon2}
\limsup_{K\to\infty}\,\norm{u_{\alpha,p}-{u_\eta^K}}_{L^2(Q)}^2+\alpha_KTV_{p_K}(u_{\alpha,p})=\norm{u_{\alpha,p}-u_\eta}_{L^2(Q)}^2+\alpha TV_p(u_{\alpha,p}),
\ee
where we used the fact that $u_\eta^K\to u_\eta$ in $L^2$ and Proposition \ref{new_equal}. Hence, by \eqref{multi_conv_recon1} and \eqref{multi_conv_recon2} we obtain that 
\be
\norm{u_{\alpha,p}-u_\eta}_{L^2(Q)}^2+\alpha TV_p(u_{\alpha,p})\geq \norm{w-u_\eta}_{L^2(Q)}^2+\alpha TV_p(w).
\ee
Therefore, we must have $u_{\alpha,p} = w$ since the minimizer of \eqref{scheme_B2_BV_p_main} is unique, which concludes \eqref{multi_conv_recon3} as desired.\\\\
We next claim \eqref{multi_conv_recon_TV}. Indeed, the liminf inequality 
\be\label{multi_conv_recon10}
\liminf_{K\to\infty} TV_{p_K}(u^K_{\alpha_K,p_K})\geq TV_p(u_{\alpha,p})
\ee
can be directly obtained from \eqref{star}. Again, since $u_{\alpha_K,p_K}^n$ is the unique minimizer, we observe that 
\be
\norm{u^K_{\alpha_K,p_K}-u_\eta^K}_{L^2(Q)}^2+\alpha_K TV_{p_K}(u^K_{\alpha_K,p_K})\leq \norm{u_{\alpha,p}-u_\eta^K}_{L^2(Q)}^2+\alpha_K TV_{p_K}(u_{\alpha,p}).
\ee
By \eqref{multi_conv_recon3} we infer that 
\be
\lim_{K\to\infty}\norm{u^K_{\alpha_K,p_K}-u_\eta^K}_{L^2(Q)}^2 = \norm{u_{\alpha,p}-u_\eta}_{L^2(Q)}^2,
\ee
which, in turn, yields
\begin{align}
&\norm{u_{\alpha,p}-u_\eta}_{L^2(Q)}^2+\alpha TV_p(u_{\alpha,p})\\
&\leq \lim_{K\to\infty}\norm{u^K_{\alpha_K,p_K}-u_\eta^K}_{L^2(Q)}^2+\limsup_{K\to\infty}\alpha_K TV_{p_K}(u^K_{\alpha_K,p_K})\\
&\leq  \limsup_{K\to\infty}\fmp{\norm{u_{\alpha,p}-u_\eta^K}_{L^2(Q)}^2+\alpha_K TV_{p_K}(u_{\alpha,p})}\\
&\leq\norm{u_{\alpha,p}-u_\eta}_{L^2(Q)}^2+\alpha TV_p(u_{\alpha,p}),
\end{align}
where at the last inequality we used Proposition \ref{new_equal} again. Thus, we conclude that 
\be
\limsup_{K\to\infty}\alpha_K TV_{p_K}(u^K_{\alpha_K,p_K})= \alpha TV_p(u_{\alpha,p}).
\ee
This, together with \eqref{multi_conv_recon10}, we conclude \eqref{multi_conv_recon_TV} and hence the thesis for the case $\alpha>0$.\\\\
Lastly, we assume $\alpha=0$. In this case we have $u_{0,p}=u_\eta$, and we could refer to the proof used in Theorem \ref{main_scheme_K_result} to conclude our thesis.
\end{proof}

\begin{proof}[{Proof of Theorem \ref{article_result_finite_resolution}}]
The Assertion 1 can be directly deduced from Proposition \ref{approx_bilevel}. We focus on claiming Assertion 2. Let $u_{\alpha,p}$ and $u_{\alpha,p}^K$ be obtained from \eqref{scheme_B2_BV_p_main} and \eqref{L2_resolution_K}, respectively. By the optimality condition, we have 
\be
u_{\alpha,p}-u_\eta = -\alpha \partial TV_p(u_{\alpha,p})
\ee
and 
\be
u_{\alpha,p}^K-u_{\eta}^K = -\alpha\partial TV_p(u_{\alpha,p}^K).
\ee
Subtracting one from another, we deduce that
\be
u_{\alpha,p}-u_{\alpha,p}^K+u_{\eta}^K-u_\eta=-\alpha\fmp{\partial TV_p(u_{\alpha,p})-\partial TV_p(u_{\alpha,p}^K)}.
\ee
Hence, by multiplying $u_{\alpha,p}-u_{\alpha,p}^K$ on the both hand side and integrating over $Q$, we have that 
\begin{align}
&\norm{u_{\alpha,p}-u_{\alpha,p}^K}_{L^2(Q)}^2+\fjp{u_{\alpha,p}-u_{\alpha,p}^K,u_{\eta}^K-u_\eta}_{L^2}\\
 &= -\alpha\fjp{\partial TV_p(u_{\alpha,p})-\partial TV_p(u_{\alpha,p}^K),u_{\alpha,p}-u_{\alpha,p}^K}_{L^2}\\
 &\leq 0,
\end{align}
which yields that
\be
\norm{u_{\alpha,p}-u_{\alpha,p}^K}_{L^2(Q)}\leq \norm{u_{\eta}^K-u_\eta}_{L^2(Q)}.
\ee
We conclude our thesis by following the argument used in the proof of Theorem \ref{article_result}.
\end{proof}
In \cite{liu2014ANISOTRIPIC}, it is shown that if the noisy image $u_\eta$ is a piece-wise constant function, we can take advantage of this when numerically computing the solution $u_{\alpha,p}$.
One good choice for the approximation sequence $u_\eta^K$, therefore, could be the piece-wise average of $u_\eta$ introduced as follows.
\begin{define}[Piece-wise approximation of the noisy image]
Let $Q=(0,1)^N$ and the corrupted image $u_\eta\in L^2(Q)$ be given. We define the $K$-resolution approximation $u_\eta^K$ of $u_\eta$ via its average
\be\label{finite_represent}
u_\eta^K(x):=\fint_{Q_K\fsp{k_1,\ldots,k_N}} {u_\eta}\,dx\text{ for }x\in Q_K\fsp{k_1,\ldots,k_N},
\ee  
where 
\be
Q_K\fsp{k_1,\ldots,k_N}:=[k_1-1/K,k_1/K]\times\cdots\times \fmp{k_N-1/K,k_N/K},
\ee
for $1\leq k_1,\ldots,k_N\leq K$.
\end{define}
We note that $u_\eta^K$ defined in \eqref{finite_represent} satisfies \eqref{smooth_approx_u_eta}.\\\\
As a result of Theorem \ref{article_result} and Theorem \ref{article_result_finite_resolution}, the following corollary can be established.
\begin{corollary}\label{main_combine_coro}
Let $u_c\in BV(Q)$ and $u_\eta\in L^2(Q)$ be given. Then, for arbitrary $\e>0$, there exists $l_\e\in \N$ large enough such that 
\be
\abs{\CC(\alpha_{\mathbb T_{l_\e}}, p_{\mathbb T_{l_\e}})-\CC(\alpha_\T,p_\T)}\leq \e.
\ee
\end{corollary}
\begin{proof}
Let $\e>0$ be given. Then by Theorem \ref{article_result_finite_resolution}, we could choose $K_\e\in \N$ large enough so that 
\be\label{main_combine_coro_eq}
\abs{\CC^{K_\e}(\alpha_\T^{K_\e}, p_\T^{K_\e})-\CC( \alpha_\T, p_\T)}\leq  2\norm{u_{\eta}^{K_\e}-u_\eta}_{L^2(Q)}\leq \frac14\e.
\ee
On the other hand, by Theorem \ref{article_result}, we have that 
\be
\abs{\CC^{K_\e}(\alpha_{\T_l}^{K_\e}, p_{\T_l}^{K_\e})-\CC^{K_\e}( \alpha_\T^{K_\e}, p_\T^{K_\e})}\leq  \sqrt{\alpha^K_U}\fmp{{1/l}^{1/2}+2{\fsp{1-N^{-1/\sqrt l}}}^{1/2}}[TV_1(u^K_\eta)]^{1/2},
\ee
where by \eqref{finite_represent} we observe that $TV_1(u^K_\eta)<+\infty$ even if $u_\eta\notin BV(Q)$.\\\\
Hence, by taking $l_\e\in\N$ large enough, and together with \eqref{main_combine_coro_eq}, we conclude
\begin{align}
&\abs{\CC(\alpha_{\mathbb T_{l_\e}}, p_{\mathbb T_{l_\e}})-\CC(\alpha_\T,p_\T)}\\
&\leq \abs{\CC(\alpha_{\mathbb T_{l_\e}}, p_{\mathbb T_{l_\e}})-\CC^{K_\e}(\alpha_{\T_{l_\e}}^{K_\e}, p_{\T_{l_\e}}^{K_\e})}+\abs{\CC^{K_\e}(\alpha_\T^{K_\e}, p_\T^{K_\e})-\CC(\alpha_\T,p_\T)}\\
&\,\,\,\,\,\,+\abs{\CC^{K_\e}(\alpha_{\T_{l_\e}}^{K_\e}, p_{\T_{l_\e}}^{K_\e})-\CC^{K_\e}(\alpha_\T^{K_\e}, p_\T^{K_\e})}\\
&\leq \frac12\e+\frac12\e\leq \e,
\end{align}
as desired.
\end{proof}

Then, based on Corollary \ref{main_combine_coro}, we suggest the following practical strategy for computing an acceptable solution $u_{\alpha_\Tl, p_\Tl}$. \\\\
\noindent\fbox{%
    \parbox{\textwidth}{%
Let $u_c\in BV(Q)$ and $u_\eta\in L^2(Q)$ be given. Let an acceptable error $\e>0$ be given.
\begin{itemize}
\item
Initialization: Compute $\alpha_U$ defined in Proposition \ref{prop:last_word_go_upper} and construct piece-wise constant function $u_\eta^K$, defined in \eqref{finite_represent}, such that 
\be
\norm{u_\eta^K-u_\eta}_{L^2(Q)}\leq \frac14\e
\ee
is satisfied.
\item
Step 1: Submit $u_\eta^K$ into Theorem \ref{article_result}, and determine step $l\in\N$ so that the right hand side of \eqref{intro_estimation_last} less than $\e/4$.
\item
Step 2: Determine one optimal solution $(\alpha_\Tl, p_\Tl)\in\mathbb A[\T_l]$. By Theorem \ref{article_result} and Theorem \ref{article_result_finite_resolution} we have that
\be
\abs{\CC(\alpha_\Tl, p_\Tl)-\CC(\alpha_\T,p_\T)}\leq \e.
\ee
\item
Step 3: The reconstructed image $u_{\alpha_\Tl, p_\Tl}$ is then an acceptable optimal solution defined in \eqref{accptable_train_result_intro}.
\end{itemize}
}%
}

\section{Numerical simulations and conclusions}\label{sec_NSC}
\subsection{Simulations and insights}
We perform numerical simulations of the bilevel scheme $\mathcal T$ using the clean image $u_c$ and the noisy image $u_\eta$ shown in the first and second picture in Figure \ref{fig:clean_noise}, respectively, and we report that their total variations $TV_1(u_\eta)=264.5255$ and $TV_1(u_c)=98.4627$. 
Note, that we are reporting the TV values of finite resolution digital images which coincide with their piece-wise constant approximation mentioned in Section \ref{relaxation_corrupted_sec}. Ideally, a clean image $u_c\in BV(Q)$ can only be captured by a ``super" camera which has infinite resolution. However, in the real world, such ``super" camera, with infinite resolution, does not exist, and hence, in the numerical section, we assume that a finite $K\in\N$ resolution clean image $u_c^K$ that we wish to capture by a real world digital camera is a piecewise constant function, which is related to $u_c$ via its averages and defined in the way of \eqref{finite_represent}, with $u_c$ in place of $u_\eta$.\\\\
The principal sources of noise in digital images are introduced during acquisition, for example, the sensor noise caused by poor illumination, circuity of a scanner, and the unavoidable shot noise of a photon detector. The noise is only generated during the acquiring of the image, i.e., it is only added to $u_{c}^K$; and each time we acquire an image, we produce a different noise $\eta^K$ . Therefore, we propose to use a piecewise constant function $\eta^K$ over $Q_K$ to represent the noise at the resolution level $K\in\N$, and we write
\be
u_\eta^K :=u_c^K +\eta_K.
\ee
That is, when a image is taken with resolution $K\in\N$, although we only wish to observe $u_c^K$ , the noise $\eta^K$ is an unavoidable by-product, and hence the corrupted image $u_\eta^K$ is produced.\\\\
Therefore, we assume that all (corrupted) images captured by digital camera with finite resolution is already a piece-wise constant function, which implies that $u_\eta\in BV(Q)$, and hence no relaxation is needed as studied in Section \ref{relaxation_corrupted_sec}.

\begin{figure}[!h]
  \centering
        \includegraphics[width=1.0\linewidth]{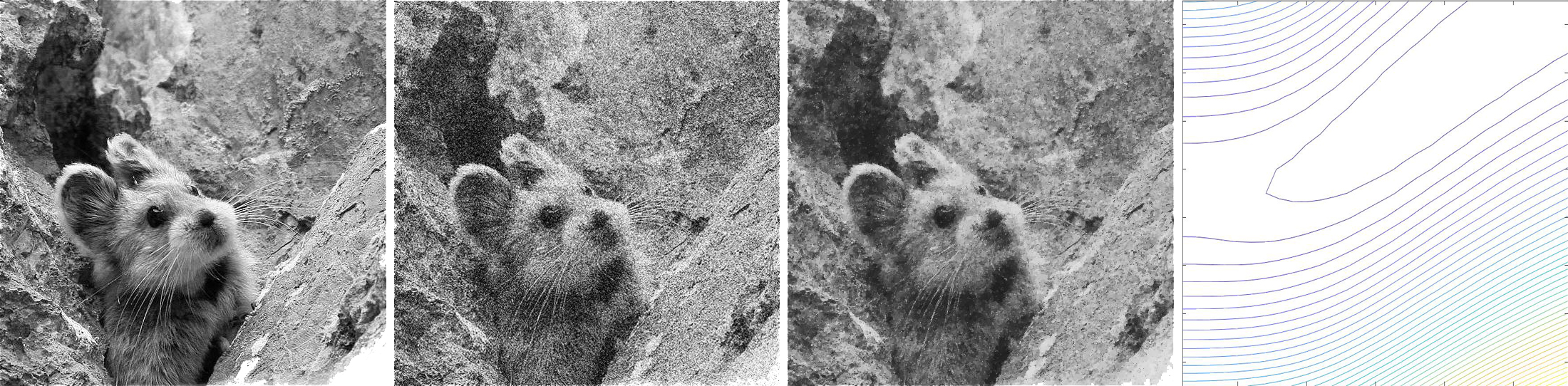}
\caption{L-R: Clean image; noisy version (with artificial Gaussian noise); optimally reconstructed image $u_{\alpha_\T, p_\T}$ provided by the Scheme $\mathcal T$ \eqref{scheme_B1_BV_p_main}-\eqref{scheme_B2_BV_p_main}, where $p_\T=2.6622$; and the contour plot of the assessment function $\mathcal A(\alpha,p)$ at $\T_{l}$, $l=10000$. We note that $\mathcal A(\alpha,p)$ is indeed not quasi-convex.}
\label{fig:clean_noise}
\end{figure}

The Level 2 problem \eqref{scheme_B2_BV_p_main} is solved via the primal-dual algorithm studied in \cite{chambolle2011first, chambolle2017stochastic}, as we can recast \eqref{scheme_B2_BV_p_main} as 
\be
\min{\max\flp{-\alpha\fjp{u,\divg\vp}+\norm{u-u_\eta}_{L^2}^2-\delta_{V_{p^\ast}}(v):\,\, v\in C_c^\infty(Q;\R^2),\,\,u\in L^2(Q)}}.
\ee
Here $\delta_{V_{q^{\ast}}}$ denotes the indicator function of the set $V_q$ 
\be
V_{q}:=\flp{v\in C_c^\infty:\,\,\abs{v}_{\ell^q}\leq 1}.
\ee
For the sake of appropriate comparison, we use the finite training ground $\T_l$ at $l=10000$ to simulate the continuous training ground $\T$, and we plot the contour image of $\CC(\alpha,p)$ at $\T_{l}$, $l=10000$, in the last column in Figure \ref{fig:clean_noise}. We summarize our simulation results and computed $l$'s from Theorem \ref{article_result} in Table \ref{table_test_result}. Note that the step $l$ predicted by Theorem \ref{article_result} (shown in column 2) is rounded up to the nearest integer.

\begin{table}[!h]
\begin{tabular}{|l|l|l|l|l|l|}
\hline
Acceptable error & $l$ estimated by Theorem \ref{article_result} & Numerical error  & Optimal $(\alpha_\T,p_\T)$ \\ \hline
$\e=0.2$ & $l=37$  & 0.1018   & (0.04,\,2.26) \\ \hline
$\e=0.1$ &$l=472$  &0.0217  &  (0.042,\,2.263)   \\ \hline
 $\e=0.05$&$l=7319$ & 0.0153 & (0.0488,\,2.2622)     \\ \hline
\end{tabular}
\caption{Simulation for the example in Figure \ref{fig:clean_noise} for different discretisation levels $l$ selected from Theorem \ref{article_result} with three acceptable errors.  We also report that $TV_{2.6622}(u_c)=14.6654$ while $TV_{2.2622}(u_{0.0488,\,2.6622})=11.0291$. }
\label{table_test_result}
\end{table}
We also applied the estimate in Theorem \ref{article_result} for the bilevel scheme $\mathcal B$ and we observe for the example in Figure \ref{fig:clean_noise} that
\be
\inf\CC(\alpha,p)<\inf\CC(\alpha,2),
\ee
which indicates that the scheme $\mathcal T$ in which we optimise over the parameter $p$ indeed provides an improved reconstruction result compared with the $TV_2$ scheme $\mathcal B$.

\begin{figure}[!h]
  \centering
        \includegraphics[width=1.0\linewidth]{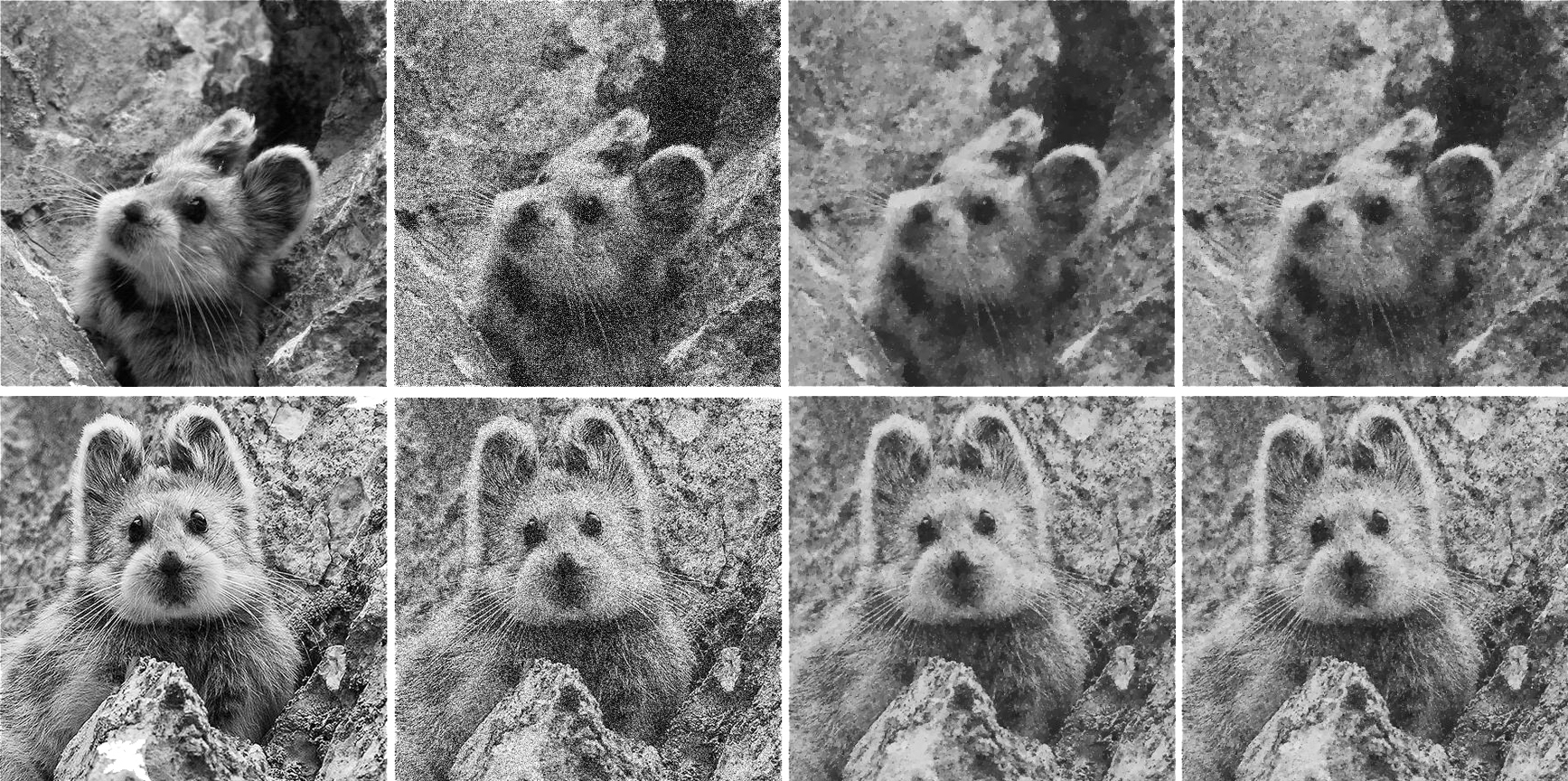}
\caption{Left column to right column: test images; noisy version (with artificial Gaussian noisy); optimal reconstructed image $u_{\alpha_\T,2}$ provided by Scheme $\mathcal B$ \eqref{scheme_B1_BV}-\eqref{scheme_B2_BV} and optimal reconstructed image $u_{\alpha_\T, p_\T}$ provided by Scheme $\mathcal T$ \eqref{scheme_B1_BV_p_main}-\eqref{scheme_B2_BV_p_main}, respectively. We report that the optimal $p$ are achieved at $p_\T=1.0197$ and $p_\T=1.1359$ for test set in row one and two, respectively. 
}
\label{fig:clean_noise2}
\end{figure}

\subsection{Conclusions and future works}
In this work, we first constructed $\ell^p$-(an)-isotropic total variation $TV_{p}$ semi-norms and applied it into the imaging processing problems. This class of semi-norms can be viewed as a generalization of the standard total variation. Then, we introduce a semi-supervised learning scheme to optimize the underlying Euclidean parameter $p$ in $TV_{p}$. A further finite approximation method of such learning scheme allows us not only conclude the existence of global optimization but also allows us to numerically compute it, especially in the situation that the convex condition is missing.\\\\
We also want to remark a few words about the inefficiency of the finite approximation scheme studied above, and provide several ways to mitigate such inefficiency. The finite approximation scheme searches the global optimizer by walking through every grid point. Although the number of grid points is finite, the massive amount of them would inevitably cause long CPU time. One way to mitigate such problem is to implement a parallel computation method as the construction and searching procedure used in our finite approximation is in particular suitable for such acceleration method.\\\\
On the other hand, we observe from Table \ref{table_test_result} that the numerical error (column 3) (which is what we actually obtained) is much smaller than the given acceptable error (column 1) (which is what we expect to obtain), which likely causes over-computing and hence wait of CPU time. This phenomenon is partially due to the large value of $TV(u_\eta)$ and $\alpha_U$. To mitigate this drawback, we observe  from the numerical simulation, reported in Table \ref{table_compare_tv}, that the optimal solution usually has the property that 
\be\label{almost_re_clean}
TV_{p_\T}(u_{\alpha_\T, p_\T}) \approx TV_{ p_\T}(u_c).
\ee

\begin{table}[h]
\begin{tabular}{|l|l|l|l|l|}
\hline
Test image & optimal $p_\T$ & $TV_{p_\T}(u_\eta)$ & $TV_{p_\T}(u_c)$ & $TV_{p_\T}(u_{\alpha_\T,p_\T})$ \\ \hline
Figure \ref{fig:clean_noise} & 2.6622          & 29.7651        & 14.6654      & 11.0291        \\\hline
Figure \ref{fig:clean_noise2}, row 1 & 1.0197          & 142.0373        & 98.4627      & 108.0941         \\\hline
Figure \ref{fig:clean_noise2}, row 2 & 1.1359          & 102.3078        & 48.7725      & 33.9762 \\ \hline      
\end{tabular}
\caption{Total variation of clean image, corrupted image, and optimal reconstructed image.}
\label{table_compare_tv}
\end{table}

Thus, if we take \eqref{almost_re_clean} for guaranteed, we could reduce the range of regularization parameters $\alpha$ to be $[\alpha_L,\alpha_R]\subset [0,\alpha_U]$ such that 
\be
TV_{p}(u_{\alpha_L,p})\geq TV_{p}(u_c)+\frac12(TV_{p}(u_\eta)-TV_{p}(u_c))>\frac12 TV_p(u_c)\geq TV_{p}(u_{\alpha_R,p}).
\ee
Then, we can reduce estimate \eqref{intro_estimation_last} in Theorem \ref{article_result} to 
\be
\abs{\CC(\alpha_\Tl, p_\Tl)-\CC(\alpha_\T,p_\T)}\leq  \sqrt{\alpha_R-\alpha_L}\fmp{\abs{1/l}^{1/2}+2\fsp{\fsp{1-N^{-1/\sqrt l}}}^{1/2}}(TV_1(u_c))^{1/2}.
\ee
We notice that this new estimate uses only the total variation of the clean image $u_c$, which is assumed to be much smaller than the variation of the corrupted image and the value of $\alpha_R-\alpha_L$ is also much smaller than $\alpha_U$. However, due to the length of this article and the fact that the estimate \eqref{almost_re_clean} already fully satisfies our purpose, we decide not to pursue on how to prove \eqref{almost_re_clean} but leave it for future work. \\\\
Another interesting direction is to understand which properties of a given image influence the value of optimal tuning parameter $\alpha$ and underlying Euclidean norm $p$ the most. The tuning parameter, by its definition, decides the regularization strength, and hence, higher noise usually requires a larger $\alpha$ value (as an extreme example, for an image with zero noise the optimal $\alpha$ is $0$).  However, what properties of a given image decide the optimal value for the underlying Euclidean norm $p$ is unclear so far. As we can see from Table \ref{table_compare_tv} for the 3 test images (with the exact same level of Gaussian noise) the optimal value $p$ ranges from almost 1 to almost 3. The current guess is that the optimal $p$ is partially decided by the properties of edges of the given image but a detailed theoretical explanation is still missing. \\\\
As a final remark of the training scheme $\mathcal T$ with $\ell^p$-(an)-isotropic total variation, the introduction of Euclidean order $p\in[1,+\infty]$ into training scheme only meant to expand the training choices, but not to provide a superior seminorm to the popular choice $TV_2$ or $TV_1$. The optimal order $\tilde p\in\flp{1,2}$ or not, is completely up to the given training image $u_\eta=u_c+\eta$.\\\\
\textbf{Acknowledgments.} PL acknowledges support from the EPSRC Centre Nr. EP/N014588/1 and the Leverhulme Trust project on Breaking the non-convexity barrier. CBS acknowledges support from the Leverhulme Trust project on Breaking the non-convexity barrier, the Philip Leverhulme Prize, the EPSRC grant Nr. EP/M00483X/1, the EPSRC Centre Nr. EP/N014588/1, the RISE projects CHiPS and NoMADS, the Cantab Capital Institute for the Mathematics of Information and the Alan Turing Institute. We gratefully acknowledge the support of NVIDIA Corporation with the donation of a Quadro P6000 GPU used for this research.
\bibliographystyle{abbrv}
\bibliography{PTV_T_VReady}{}

\begin{thebibliography}{10}

\bibitem{ambrosio2000functions}
L.~Ambrosio, N.~Fusco, and D.~Pallara.
\newblock {\em Functions of bounded variation and free discontinuity problems}.
\newblock Oxford Mathematical Monographs. The Clarendon Press, Oxford
  University Press, New York, 2000.

\bibitem{MR2067663}
P.~L. Antonelli, editor.
\newblock {\em Handbook of {F}insler geometry. {V}ol. 1, 2}.
\newblock Kluwer Academic Publishers, Dordrecht, 2003.
\newblock With 1 CD-ROM containing the software package FINSLER.

\bibitem{braides2002gamma}
A.~Braides.
\newblock {\em {$\Gamma$}-convergence for beginners}, volume~22 of {\em Oxford
  Lecture Series in Mathematics and its Applications}.
\newblock Oxford University Press, Oxford, 2002.

\bibitem{chambolle2017stochastic}
A.~Chambolle, M.~J. Ehrhardt, P.~Richt{\'a}rik, and C.-B. Sch{\"o}nlieb.
\newblock Stochastic primal-dual hybrid gradient algorithm with arbitrary
  sampling and imaging application.
\newblock {\em arXiv preprint arXiv:1706.04957}, 2017.

\bibitem{chambolle2011first}
A.~Chambolle and T.~Pock.
\newblock A first-order primal-dual algorithm for convex problems with
  applications to imaging.
\newblock {\em Journal of Mathematical Imaging and Vision}, 40(1):120--145,
  2011.

\bibitem{chen2013revisiting}
Y.~Chen, T.~Pock, R.~Ranftl, and H.~Bischof.
\newblock Revisiting loss-specific training of filter-based mrfs for image
  restoration.
\newblock In {\em Pattern Recognition}, pages 271--281. Springer, 2013.

\bibitem{chen2014insights}
Y.~Chen, R.~Ranftl, and T.~Pock.
\newblock Insights into analysis operator learning: From patch-based sparse
  models to higher order mrfs.
\newblock {\em IEEE Transactions on Image Processing}, 23(3):1060--1072, March
  2014.

\bibitem{MR1201152}
G.~Dal~Maso.
\newblock {\em An introduction to {$\Gamma$}-convergence}, volume~8 of {\em
  Progress in Nonlinear Differential Equations and their Applications}.
\newblock Birkh\"auser Boston, Inc., Boston, MA, 1993.

\bibitem{de2013image}
J.~C. De~los Reyes and C.-B. Sch{\"o}nlieb.
\newblock Image denoising: learning the noise model via nonsmooth
  {PDE}-constrained optimization.
\newblock {\em Inverse Probl. Imaging}, 7(4):1183--1214, 2013.

\bibitem{reyes2015structure}
J.~C. De~Los~Reyes, C.-B. Sch{\"o}nlieb, and T.~Valkonen.
\newblock The structure of optimal parameters for image restoration problems.
\newblock {\em J. Math. Anal. Appl.}, 434(1):464--500, 2016.

\bibitem{domke2012generic}
J.~Domke.
\newblock Generic methods for optimization-based modeling.
\newblock In {\em AISTATS}, volume~22, pages 318--326, 2012.

\bibitem{MR1727362}
I.~Ekeland and R.~T\'emam.
\newblock {\em Convex analysis and variational problems}, volume~28 of {\em
  Classics in Applied Mathematics}.
\newblock Society for Industrial and Applied Mathematics (SIAM), Philadelphia,
  PA, english edition, 1999.
\newblock Translated from the French.

\bibitem{Engl:2000aa}
H.~W. Engl, M.~Hanke, and A.~Neubauer.
\newblock {\em Regularization of Inverse Problems}, volume 375 of {\em
  Mathematics and Its Applications}.
\newblock Springer Verlag, 2000.

\bibitem{evans2015measure}
L.~C. Evans and R.~F. Gariepy.
\newblock {\em Measure theory and fine properties of functions}.
\newblock Textbooks in Mathematics. CRC Press, Boca Raton, FL, revised edition,
  2015.

\bibitem{MR2496060}
T.~Goldstein and S.~Osher.
\newblock The split {B}regman method for {$L1$}-regularized problems.
\newblock {\em SIAM J. Imaging Sci.}, 2(2):323--343, 2009.

\bibitem{Golub:1979aa}
G.~H. Golub, M.~Heat, and G.~Wahba.
\newblock Generalized cross validation as a method for choosing a good ridge
  parameter.
\newblock {\em Technometrics}, 21:215--223, 1979.

\bibitem{MR2727338}
M.~Grasmair and F.~Lenzen.
\newblock Anisotropic total variation filtering.
\newblock {\em Appl. Math. Optim.}, 62(3):323--339, 2010.

\bibitem{MR1984880}
E.~Haber and L.~Tenorio.
\newblock Learning regularization functionals---a supervised training approach.
\newblock {\em Inverse Problems}, 19(3):611--626, 2003.

\bibitem{Hansen:1992aa}
P.~C. Hansen.
\newblock Analysis of discrete ill-posed problems by means of the {L-curve}.
\newblock {\em SIAM Review}, 34:561--580, 1992.

\bibitem{MR1819784}
K.~C. Kiwiel.
\newblock Convergence and efficiency of subgradient methods for quasiconvex
  minimization.
\newblock {\em Math. Program.}, 90(1, Ser. A):1--25, 2001.

\bibitem{kunisch2013bilevel}
K.~Kunisch and T.~Pock.
\newblock A bilevel optimization approach for parameter learning in variational
  models.
\newblock {\em SIAM J. Imaging Sci.}, 6(2):938--983, 2013.

\bibitem{doi:10.1137/16M1103610}
M.~{\L}asica, S.~Moll, and P.~B. Mucha.
\newblock Total variation denoising in {$l^1$} anisotropy.
\newblock {\em SIAM J. Imaging Sci.}, 10(4):1691--1723, 2017.

\bibitem{liu2014ANISOTRIPIC}
P.~Liu and C.-B. Sch{\"o}nlieb.
\newblock An-isotropic total variation and piecewise constant solutions.
\newblock {\em in preparation}, 2019.

\bibitem{meyer2001oscillating}
Y.~Meyer.
\newblock {\em Oscillating patterns in image processing and nonlinear evolution
  equations: the fifteenth Dean Jacqueline B. Lewis memorial lectures},
  volume~22.
\newblock American Mathematical Soc., 2001.

\bibitem{moll2005anisotropic}
J.~Moll.
\newblock The anisotropic total variation flow.
\newblock {\em Mathematische Annalen}, 332(1):177--218, 2005.

\bibitem{Morozov:1966aa}
V.~A. Morozov.
\newblock On the solution of functional equations by the method of
  regularization.
\newblock {\em Soviet mathematics -- Doklady}, 7:414--417, 1966.

\bibitem{rudin1992nonlinear}
L.~I. Rudin, S.~Osher, and E.~Fatemi.
\newblock Nonlinear total variation based noise removal algorithms.
\newblock {\em Phys. D}, 60(1-4):259--268, 1992.
\newblock Experimental mathematics: computational issues in nonlinear science
  (Los Alamos, NM, 1991).

\bibitem{strong1996exact}
D.~M. Strong, T.~F. Chan, et~al.
\newblock Exact solutions to total variation regularization problems.
\newblock In {\em UCLA CAM Report}. Citeseer, 1996.

\bibitem{tappen2007learning}
M.~F. Tappen, C.~Liu, E.~H. Adelson, and W.~T. Freeman.
\newblock Learning gaussian conditional random fields for low-level vision.
\newblock In {\em 2007 IEEE Conference on Computer Vision and Pattern
  Recognition}, pages 1--8, June 2007.

\bibitem{MR1666943}
J.~Weickert.
\newblock {\em Anisotropic diffusion in image processing}.
\newblock European Consortium for Mathematics in Industry. B. G. Teubner,
  Stuttgart, 1998.

\end{thebibliography}
\end{document}